\documentclass{amsart}[12pt]
\usepackage{amssymb, amsmath}
\usepackage{amsthm}
\usepackage{float,xcolor}
\usepackage{float,geometry}
\usepackage{tikz-cd,tikz,pgfplots, mathdots}
\input{xy}
\xyoption{poly}
\xyoption{2cell}
\xyoption{all}

\theoremstyle{definition}

\newcommand{\disp}{\displaystyle}

\newcommand{\C}{\mathbb{C}}

\newcommand{\Q}{\mathbb{Q}}
\newcommand{\Z}{\mathbb{Z}}

\newcommand{\G}{\mathcal{G}}
\newcommand{\x}{\textbf{\textit{x}}}

\renewcommand{\phi}{\varphi}
\newtheorem*{theorem*}{Theorem}
\newtheorem{theorem}{Theorem}[section]
\newtheorem{lemma}[theorem]{Lemma}
\newtheorem{example}[theorem]{Example}
\newtheorem{corollary}[theorem]{Corollary}
\newtheorem{remark}[theorem]{Remark}

\newtheorem{proposition}[theorem]{Proposition}
\newtheorem{definition}[theorem]{Definition}

\newcommand{\M}{\mathcal{M}}
\newcommand{\F}{\mathcal{F}}

\renewcommand{\P}{\mathcal{P}}
\newcommand{\Gr}{\mathrm{Gr}}
\newcommand{\FR}{\mathbb{FR}}
\newcommand{\slk}{\mathbb{SL}_k}

\title{$SL_k$-Tilings and Paths in $\Z^k$}
\author{Zachery Peterson}
\address{College of William \& Mary, Department of Mathematics, Jones Hall, 200 Ukrop Way, Williamsburg, VA 23187-8795, USA }

\author{Khrystyna Serhiyenko}
\address{University of Kentucky, Lexington, Department of Mathematics, 951 Patterson Office
Tower, Lexington, KY 40506-0027, USA}

\begin{document}

	\begin{abstract}
	An $SL_k$-tiling is a bi-infinite array of integers having all adjacent $k\times k$
minors equal to one and all adjacent $(k+1)\times (k+1)$ minors equal to zero.  Introduced and studied by Bergeron and Reutenauer, $SL_k$-tilings generalize the notion of Conway-Coxeter frieze patterns in the case $k=2$. In a recent paper, Short showed a bijection between bi-infinite paths of reduced rationals in the Farey graph and $SL_2$-tilings.  We extend this result to higher $k$ by constructing a bijection between $SL_k$-tilings and certain pairs of bi-infinite strips of vectors in $\Z^k$ called paths.  The key ingredient in the proof is the connection to Pl\"ucker friezes and Grassmannian cluster algebras.  As an application, we obtain results about periodicity, duality, and positivity for tilings.

\end{abstract}
\maketitle
	\section{Introduction}\label{sec:introduction}

A \textit{frieze} is a bi-infinite array of integers satisfying the so-called diamond rule.
    Friezes were first introduced and studied in the case $k=2$ by Conway and Coxeter in 1970's \cite{cc73}\cite{cc732}.
    They proved various facts about periodicity and symmetry of friezes and, in particular, that there is a bijection between $SL_2$-friezes over positive integers, often called \textit{Conway-Coxeter friezes}, and triangulations of polygons.
    Later, the discovery of cluster algebras in 2000's \cite{fz02} created a newfound interest in friezes, as  cluster algebras of type A are closely tied to triangulations of polygons. 
    Moreover, additive categorification of cluster algebras yields another important relation between friezes and representation theory of quivers, first shown in \cite{cc06}.
    Subsequent connections were also discovered between frieze patterns and Farey graphs, cross-ratios, and continued fractions, see the survey \cite{mg15} and references therein.
    From this foundation came a series of generalizations of $SL_2$-friezes that lead to a whole body
of work on the topic in the case $k= 2$. Some examples include friezes coming from triangulations of
annuli or once-punctured disks \cite{BM09, BPT16}, friezes with entries in other rings \cite{CHJ, HJ20}, friezes satisfying
variants of the diamond rule \cite{CJ21, FL21, MGOT15}, and mesh friezes coming from quiver representations \cite{bfgst21}.  

In another direction, the notion of $SL_2$-friezes was extended to higher values of $k$, which can be thought of as dimension.
Subsequently, $SL_k$-friezes were found to be related to linear difference equations, the
Gale duality, and $T$-systems from physics, see \cite{mgost13}.  Furthermore, these higher friezes are  closely connected to the coordinate ring of the Grassmannian $\Gr(k, n)$.
    The Grassmannian $\Gr(k, n)$ is a projective variety via the Pl\"ucker embedding, whose homogeneous coordinate ring is one of the first and most well-known examples of cluster algebras \cite{s03}.
    It was shown in \cite{bfgst21} that one may instead use certain Pl\"ucker coordinates as entries in an $SL_k$-frieze, and the Pl\"ucker relations that they satisfy yield the desired diamond rule, the defining property of a frieze.  These Pl\"ucker friezes can be thought of as universal friezes, since any frieze with integer values appears as an evaluation of a Pl\"ucker frieze at a particular element of the Grassmannian. 
    Later, a generalization of $SL_k$-friezes was studied in relation to juggling functions and positroid varieties in \cite{dm22}. 
    In addition, recently \cite{c23} extended the work on triangulations of subpolygons and friezes beyond the two-dimensional case.

    In another direction, Bergeron and Reutenauer \cite{br10} further generalized friezes to $SL_k$-tilings, infinite arrays $\M=(m_{i\,j})_{i,j\in\Z}$ where the determinant of every $k\times k$ submatrix equals $1$.
    These tilings are called tame if every larger adjacent minor has a determinant of $0$.
    Additionally, they showed how one could construct a tiling from a frieze by rotating the frieze and extending it infinitely to the plane. 
    In their exploration of tame $SL_k$-tilings they show that tilings may be represented using their so-called linearization data.
    In the case of two-dimensional tilings, Short \cite{s22} relates this linearization data to sequences of rational numbers, and proves a bijection between $SL_2$-tilings and pairs of paths in the Farey graph.
    Further restrictions of this map give geometric interpretations for positive and periodic $SL_2$-tilings, as well as tilings from $SL_2$-friezes.
    This bijection was later extended to the 3D Farey graph by considering tilings over Eisenstein integers \cite{fkst23}, as well as entries in the field $\Z/n\Z$ \cite{svsz24}.

    In summary, there has been a lot of work in the case $k=2$ as well as some more recent developments in the case of $SL_k$-friezes for higher values of $k$.  In this paper, we study $SL_k$-tilings for $k\geq 2$, which are much less understood.
	We consider them from the perspective of linear algebra and obtain a generalization of Short's result \cite{s22}.
	Lacking the connection to the geometry of the Farey graph as in the case $k=2$, we introduce a new algebraic notion of a path in $\Z^k$.
	Let $\gamma=\{\gamma_i\}_{i\in\Z}$ be a bi-infinite strip of $k$-column vectors $\gamma_i\in\Z^k$ with the property that the matrix $\left(\gamma_i,\ldots,\gamma_{i+k-1} \right)$, whose columns are $k$ consecutive entries of $\gamma$ has determinant 1.
	We denote the set of all such strips $\P_k$ and we call $\gamma$ a \textit{path}.
    Additionally, we define the notion of multiplication of a path by a matrix $A\in SL_k(\Z)$ as
    \[
    A\gamma=(\cdots, A\gamma_1,A\gamma_2,A\gamma_3,\cdots)\in \P_k.
    \]
	Let $\slk$ denote the set of all tame $SL_k$-tilings, then our main result can be stated as follows. 
		
	\begin{theorem}
		The map $\Phi$ given by
		\begin{align*}
		\Phi:(\P_k\times\P_k)/SL_k(\Z)&\rightarrow\slk\\
		(\gamma,\delta)&\mapsto \M=(m_{i\,j})_{i,j\in\Z},
		\end{align*} 
		where $m_{i\,j}=\det(\gamma_i,\ldots,\gamma_{i+k-2},\delta_j)$ is a bijection between tame $SL_k$-tilings and pairs of paths modulo the diagonal action by $SL_k(\Z)$.
	\end{theorem}

	The proof relies heavily on the connection between $SL_k$-friezes and Pl\"ucker coordinates.
	We also obtain a number of other correspondences.
	The first is about the class of tilings which result specifically from (possibly infinite) friezes, and which we denote by $\FR_k$.
	
	\begin{theorem}
		The restriction of the map $\Phi$ given by
		\begin{align*}
			\Phi_\iota:\P_k/SL_k(\Z)&\rightarrow \FR_k\\
			\gamma&\mapsto \M=(m_{i\,j})_{i,j\in\Z}, 
		\end{align*}
		where $m_{i\,j}=\det(\gamma_i,\ldots, \gamma_{i+k-2},\gamma_j)$ is a bijection between tame $SL_k$-tilings from $SL_k$-friezes and equivalence classes of paths.
	\end{theorem}

	We also prove in Corollary~\ref{blockperiodic} that pairs of periodic paths are in bijection with $SL_k$-tilings which have the corresponding periods on their rows and columns.
	On the other hand, the classification of positive $SL_2$-friezes and tilings in terms of clockwise paths in the Farey graph obtained in \cite{s22}, relies heavily on the geometry of the Farey graph.	It is not clear how to interpret this geometric notion in our setting, and in the case of $SL_k$-friezes we introduce an algebraic notion of a quiddity sequence for higher values of $k$.
	Appealing to the structure of the coordiante ring of the Grassmannian $\mathrm{Gr}(k,n)$ and Pl\"ucker relations, we prove that for certain small values of $n$ and $k$, such as most cases where $n\leq 8$, positive $SL_k$-friezes are in bijection with $n$-periodic positive quiddity sequences, see Theorem~\ref{thm:positivitybij}.
	This proof requires a case-by-case analysis for different choices of $k$ and $n$.
		
	Finally, our bijection reinterprets the notion of a dual of a tiling, introduced by Bergeron and Reutenauer \cite{br10}.  The entries of the dual tiling $\M^*$ are by definition $k-1$ minors of $\M$, and this indeed yields an $SL_k$-tiling. Moreover, $(\M^*)^*$ equals $\M$ up to a shift in indices. 
        The dual can be stated in terms of the map $\Phi$ as follows, where $\widetilde{\gamma}$ denotes a certain diagonalization operation on a path $\gamma$. 
	
	\begin{theorem}
		The dual tiling $\Phi(\gamma,\delta)^*=\Phi(A\widetilde{\gamma},\widetilde{\delta})$ for some $A\in SL_k(\Z)$. 
			\end{theorem}

 In particular, since the diagonalization operation is an involution, one can easily deduce the same for the dual, whereas the original proof of this in \cite{br10} relied on some complex calculations.

 \medskip

     This paper is organized as follows.
    In Section \ref{sec:1}, we discuss the key definitions and past results relating to friezes, tilings, and Pl\"ucker coordinates.
    In Section \ref{sec:bijection} we define the map between pairs of paths and $SL_k$-tilings and prove that it is a bijection.
    In Section \ref{sec:applications} we study restrictions of this bijection to obtain periodic tilings and $SL_k$-friezes, as well as investigate connections to duality.
    In Section \ref{sec:positivity} we complete the necessary calculations to prove the partial positivity results for $SL_k$-friezes for small values of $n$ and $k$.

    \subsubsection*{Acknowledgments}
    The authors would like to thank Sophie Morier-Genoud for helpful discussions. The authors were supported by the NSF grant DMS-2054255.
	
	\section{Background}\label{sec:1}
	
	In this section we introduce the relevant background and notation.  In particular, we review the notions of friezes and tilings, and their associated linearization data. 
	
	\subsection{Friezes}
	
	We begin by defining a structure first studied by Conway and Coxeter in the case $k=2$ \cite{cc73}\cite{cc732} and then extended for higher $k$ by Bergeron and Reutenauer \cite{br10}.
	
	\begin{definition}\label{def:frieze}
		An \textit{$SL_k$-frieze} is an array of offset bi-infinite rows of integers  consisting of $k-1$ rows of zeros at the top and bottom, a row of ones below and above them, respectively, and in between $w\geq1$ rows of integers $a_{ij}$ satisfying the following properties.
		
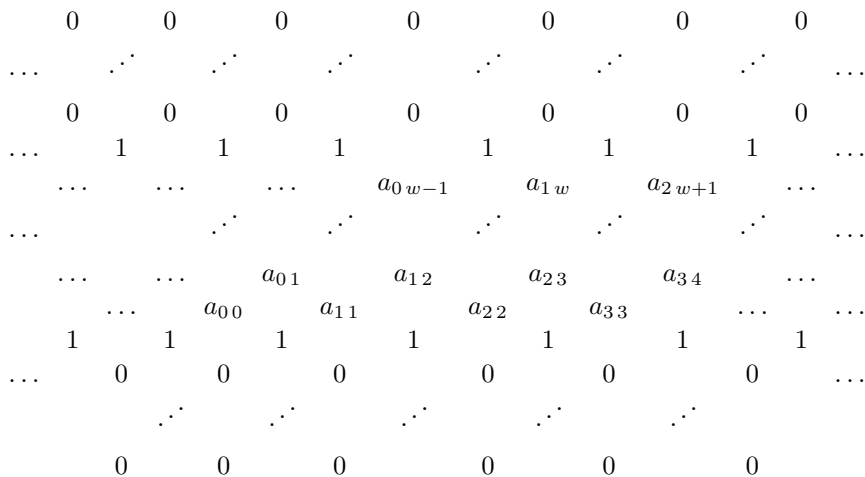
\[
\xymatrix@C=.08em@R=.06em{
 & 0 && 0 && 0 && 0 && 0 && 0 && 0  \\ 
\dots && \iddots&& \iddots &   & \iddots  && \iddots && \iddots && \iddots && \dots\\ 
& 0 && 0 && 0 && 0 && 0 && 0 && 0  \\
\dots & & 1 && 1 && 1 && 1 && 1 && 1 &&  \dots \\
 & \dots && \dots && \dots &&a_{0\,w-1} && a_{1\,w} &&  a_{2\,w+1} && \dots \\
 \dots &&&& \iddots &   & \iddots  && \iddots && \iddots && \iddots && \dots\\ 
 & \dots && \dots && a_{0\,1} && a_{1\,2} && a_{2\,3} && a_{3\,4} && \dots \\
 &  & \dots && a_{0\,0} && a_{1\,1} && a_{2\,2} && a_{3\,3} && \dots  && \dots\\
 & 1 && 1 && 1 && 1 && 1 && 1 && 1 && \\ 
\dots &  & 0 && 0 && 0 && 0 && 0 && 0 &&  \dots \\
& && \iddots&& \iddots &   & \iddots  && \iddots && \iddots &&  \\
&  &0 & & 0 && 0 &&  0 && 0 && 0 &&  &&  \\
}
\]

				\begin{enumerate}
			\item Every $k\times k$ diamond of neighboring entries has determinant $1$ when considered as a $k\times k$ matrix formed by a $45^\circ$ clockwise rotation.
			\item Every $(k+1)\times (k+1)$ diamond has determinant $0$ when considered as a matrix.
		\end{enumerate}
		We call $w$ the \textit{width} of the frieze.  We say the frieze is \textit{positive} if all the entries between the rows of ones are positive.
	\end{definition}

For an example of a frieze, see Figure \ref{friezeexample}.  We will also consider friezes with an infinite number of rows.

	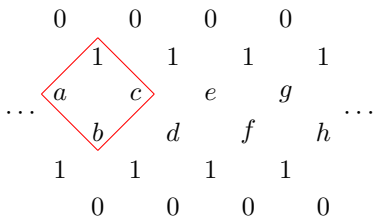
\begin{figure}
		\begin{center}
			\begin{tikzpicture}[scale=0.5]
				\foreach \x in {2,4,6,8}
				{
					\draw (\x,6) node {0};
					\draw (\x,2) node {1};
				}
				\foreach \x in {3,5,7,9}
				{
					\draw (\x,5) node {1};
					\draw (\x,1) node {0};
				}
				\draw (1,3.5) node {$\cdots$};
				\draw (10,3.5) node {$\cdots$};
				\draw (2,4) node {$a$};
				\draw (4,4) node {$c$};
				\draw (6,4) node {$e$};
				\draw (8,4) node {$g$};
				\draw (3,3) node {$b$};
				\draw (5,3) node {$d$};
				\draw (7,3) node {$f$};
				\draw (9,3) node {$h$};
				\draw[red] (1.5,4) -- (3,5.5);
				\draw[red] (3,5.5)--(4.5,4);
				\draw[red] (4.5,4) -- (3,2.5);
				\draw[red] (3,2.5)--(1.5,4);
		\end{tikzpicture}
		\caption{An $SL_2$-frieze of width $2$.  Every $2\times 2$ diamond must have determinant $1$, for example $\begin{vmatrix}
			a&1\\
			b&c
			\end{vmatrix}=ac-b=1$.}
		\label{friezeexample}
		\end{center}
		
	\end{figure}

	\begin{definition}
		An \textit{infinite $SL_k$-frieze} is an array of offset bi-infinite rows of integers  consisting of $k-1$ rows of zeros at the top followed first by a row of ones and then infinitely many rows of integers satisfying properties (1) and (2) of Definition \ref{def:frieze}.
		We say an infinite frieze is \textit{positive} if all the entries below the row of ones are positive.
	\end{definition}

	Note that a non-infinite $SL_k$-frieze may be extended to an infinite frieze with periodic rows.
	When referring to infinite friezes specifically, we will explicitly use the phrase ``infinite friezes".

	For the following definitions, fix $k,n\in\Z_{>0}$ with $k<n$.  Let $\Gr(k,n)$ denote the Grassmannian of $k$ planes in $\C^n$.  
	That is, elements of $\Gr(k,n)$ are $k$-linear subspaces of $\C^n$ which can be represented by $k\times n$ matrices of full rank.   Given a $k$-element subset $I$ of $[n]$, the {\it Pl\"ucker coordinate} $p_I$ denotes the maximal minor of a matrix on columns indexed by $I$. 	It is known that $\Gr(k,n)$ is a projective variety identified with the image of its Pl\"ucker embedding. The homogenous coordinate ring of the affine cone over the Grassmannian is denoted by 
	\[
	\mathcal{A}(k,n)=\C[\widehat{\mathrm{Gr}(k,n)}].
	\]
    Moreover, it was shown by Scott \cite{s03} that $\mathcal{A}(k,n)$ is a cluster algebra where Pl\"ucker coordinates are cluster variables.
	
	        We can also extend $p_I$ to allow for repeated entries as follows.  If $I=(i_1,\ldots,i_k)$  then $p_I(A) = \det(a_{i_1},\ldots,a_{i_k})$ where $A=(a_1,\ldots,a_n)$ is a $k\times n$ matrix. 
 Note that if there exists $p\neq q$ with $i_p=i_q$ then $p_I=0.$
	Additionally, if $(i_1,\ldots,i_k)=\pi(j_1,\ldots,j_m)$ where $\pi$ is a permutation on an ordered $k$-tuple, then $p_{\left(i_1,\ldots,i_k\right)}=\mathrm{sign}(\pi)p_{\left(j_1,\ldots,j_m\right)}$.
	Note that we want the indices of Pl\"ucker coordinates to be ordered and  elements of $[n]$.
	Thus, we introduce the following notation.
	Given a $k$-set, possibly a multiset, $\{i_1,\ldots,i_k\}$, let $j_m\equiv i_m\pmod n$ with $1\leq j_m\leq n$ for all $m\in[k]$.
	Then we write $o(i_1,\ldots,i_k):=(j_{\ell_1},\ldots,j_{\ell_k})$ where $\{j_1,\ldots,j_k\}=\{j_{\ell_1},\ldots,j_{\ell_k}\}$ and $j_{\ell_1}\leq\cdots\leq j_{\ell_k}$.
	Hence, $o(I)$ is obtained from a set $I$ by first reducing the entries mod $n$ and then reordering them in increasing order.
	
	Let $I=\{i_1,\ldots,i_{k-1}\}$ and $J=\{j_0,\ldots,j_k\}$.  
	The Pl\"ucker coordinates satisfy the \textit{Pl\"ucker relations}
	
	\begin{equation}\label{eqn:pluckerrel}
		\sum_{\ell=0}^k(-1)^\ell p_{o(I)j_\ell}\cdot p_{o(J\setminus j_\ell)}=0,
	\end{equation}
	where $o(I)j_p$ denotes the ordered tuple obtained by adjoining $j_p$ at the end of $o(I)$.
	
	We will frequently consider certain kinds of Pl\"ucker coordinates whose indices are at least partially consecutive.  We use the notation $[r]^k:= \{r,r+1,\ldots, r+k-1\}$ to denote the $k$-element set of consecutive integers starting with $r$.
	
	\begin{definition}\label{consecutive}
		The Pl\"ucker coordinate $p_{o(I)}$ is called \textit{consecutive} if $I=[r]^k$. 
		The Pl\"ucker coordinate $p_{o(I\cup\{m\})}$ is called \textit{almost consecutive} if $I=[r]^{k-1}$ and $m\in[n]\setminus I$.
	\end{definition}

	We now define a familiar structure which makes use of Pl\"ucker coordinates in place of integers.

	\begin{definition}
		The \textit{Pl\"ucker frieze} of type $(k,n)$ denoted by $\F_{(k,n)}$ is a $\Z\times[n+k-1]$ array with entries given by the map
		\begin{align*}
		(r,m)&\mapsto p_{o\left([r]^{k-1},m'\right)},
		\end{align*}
		where $m'=m+r-1$.
		For an example, see Figure \ref{pfexample}.
	\end{definition}

	\begin{figure}[H]
		\begin{center}
		\[
		\begin{array}{cccccccccc}
		&0&&0&&0&&0&&\\
		&&p_{1\,2}&&p_{2\,3}&&p_{3\,4}&&p_{4\,5}&\\
		\cdots&p_{2\,5}&&p_{1\,3}&&p_{2\,4}&&p_{3\,5}&&\cdots\\
		&&p_{3\,5}&&p_{1\,4}&&p_{2\,5}&&p_{1\,3}&\\
		&p_{3\,4}&&p_{4\,5}&&p_{1\,5}&&p_{1\,2}&&\\
		&&0&&0&&0&&0&
		\end{array}	
		\]
		\caption{The Pl\"ucker frieze $\F_{(2,5)}$ of type $(2,5)$.  When applied to a matrix satisfying the assumptions in Theorem \ref{thm:pluckerfrieze}, the consecutive Pl\"ucker coordinates in rows $2$ and $5$ become all $1$'s and the remaining entries become integers, resulting in a frieze as in Figure \ref{friezeexample}.}
		\label{pfexample}	
		\end{center}
		
	\end{figure}

	Baur, Faber, Gratz, Serhiyenko and Todorov showed that $\F_{(k,n)}$ satisfies properties (1) and (2) of Definition \ref{def:frieze}. 
	Hence, it is a frieze with entries in $\mathcal{A}(k,n)$ rather than integers \cite{bfgst21}.
	We wish to discuss their relationship to traditional friezes over the integers.
	Let $\F_{(k,n)}(A)$ denote the array of numbers $p_I(A)$ resulting from applying each Pl\"ucker coordinate in the frieze $\F_{(k,n)}$ to the $k\times n$ matrix $A$.
	
	\begin{theorem}\label{thm:pluckerfrieze}\cite[Theorem 3.1]{bfgst21}
		Let $A$ be a $k\times n$ matrix with entries in $\Z$ with the property that each consecutive Pl\"ucker coordinate has a value $1$ when applied to $A$.
		Then $\F_{(k,n)}(A)$ is an $SL_k$-frieze.
	\end{theorem}

	We call a frieze of the form $F=\F_{(k,n)}(A)$ an $SL_k$-frieze of type $(k,n)$.  Note that these friezes have $n$-periodic rows and width $w=n-k-1$.  Additionally, the following remark says that every frieze arises in this way.
	
	\begin{remark}\label{rem:pluckertofrieze}\cite[Remark 3.7]{bfgst21}
		Any $SL_k$-frieze $F$ of width $w$ over $\Z$ can be embedded into $\Gr(k,n)$ where $n=w+k+1$ as a point which can be represented by a matrix $A_{F}$ over $\Z$ whose consecutive $k\times k$ minors are ones.
		Moreover, $\F_{(k,n)}(A_F)=F$.
	\end{remark}

	We will also need the following result about particular matrices formed by Pl\"ucker coordinates.
    For $m_i\in[n]$, we use the notation $[m_1,m_2]$ for the closed cyclic interval $\{m_1,m_1+1,m_1+2,\ldots,m_2 \}$ where the elements are considered modulo $n$.
    We define open and half-open intervals similarly.
	
	\begin{definition}\label{def:pluckermatrix}
		Let $r,s\in[n]$ and $\underline{m}=(m_1,\ldots,m_s)$ with $m_i\in[n]$.
		We define the $(s\times s)$ matrix
		\[
		A_{\underline{m};r}:=(a_{i\,j})_{1\leq i,j\leq s},
		\]
		where $a_{i\,j}:=p_{o\left([r+i-1]^{k-1},m_j \right)}$ for $1\leq i,j\leq s$.
	\end{definition}

Note that any $s\times s$ diamond of neighboring entries in a Pl\"ucker frieze $\F_{(k,n)}$ yields a matrix of Pl\"ucker coordinates $A_{\underline{m};r}$ where $\underline{m}$ is consecutive. The determinant of any $A_{\underline{m};r}$ can be computed as follows. 

	\begin{proposition}\label{prop:pluckerdeterminant}\cite[Proposition 3.5]{bfgst21}
		Let $r\in[n]$, $s\in[k]$, and $\underline{m}=(m_1,\ldots,m_s)$ with $m_i\in[n]$ for all $i$ satisfying the following conditions.
		\begin{enumerate}
			\item[(c1)] $\underline{m}$ is ordered cyclically modulo $n$.
			\item[(c2)] We have $r+k-2\not\in [m_1,m_s)$.
		\end{enumerate}
		Then
		\[
		\det\left(A_{\underline{m};r}\right)=\left[\prod_{\ell=0}^{s-2}p_{o\left([r+\ell]^k\right)} \right]\cdot p_{o\left([r+s-1]^{k-s},m_1,\ldots,m_s \right)}.
		\]
	\end{proposition}

    \begin{example}
        Suppose in the case of $\Gr(3,8)$ we have $r=1$, $s=3$, and $\underline{m}=(3,4,5)$.
        Then we have
        \[
        A_{\underline{m};r}=\begin{pmatrix}
            p_{1\,2\,3} & p_{1\,2\,4} & p_{1\,2\,5} \\
            p_{2\,3\,3} & p_{2\,3\,4} & p_{2\,3\,5} \\
            p_{3\,4\,3} & p_{3\,4\,4} & p_{3\,4\,5} \\
        \end{pmatrix}=
        \begin{pmatrix}
            p_{1\,2\,3} & p_{1\,2\,4} & p_{1\,2\,5} \\
            0 & p_{2\,3\,4} & p_{2\,3\,5} \\
            0 & 0 & p_{3\,4\,5} \\
        \end{pmatrix}
        .
        \]
        By Proposition \ref{prop:pluckerdeterminant}, this determinant is
        \[\det\left(A_{\underline{m};r}\right)=\left[\prod_{\ell=0}^{1}p_{o\left([1+\ell]^3\right)} \right]\cdot p_{o\left([3]^{0},3,4,5 \right)}=p_{1\,2\,3}p_{2\,3\,4}p_{3\,4\,5}.
        \]
    \end{example}

	In the set-up of Theorem \ref{thm:pluckerfrieze}, Pl\"ucker coordinates with consecutive entries will be sent to $1$ when applied to the matrices $A$.
	Thus, we have
	\[
	\left[\prod_{\ell=0}^{s-2}p_{o\left([r+\ell]^k\right)}(A) \right]=1,
	\]
	and the determinant in Proposition \ref{prop:pluckerdeterminant} is then given by
	\[
	\det\left(A_{\underline{m};r}\right)(A)=p_{o\left([r+s-1]^{k-s},m_1,\ldots,m_s \right)}(A).
	\]

	\subsection{Tilings}

	 A generalization of friezes called \textit{tilings} was introduced and studied by Bergeron and Reutenauer in \cite{br10}.
	
	\begin{definition}
		A \textit{tiling} $\M=(m_{i\,j})_{i,j\in\Z}$ is an infinite array with values $m_{i\,j}\in\Z$.		We denote by $M_{i\,j}$ the adjacent $k\times k$ sub-matrix of $\M$ with top left entry $m_{ij}$, that is 
		\[
		M_{i\,j}=\M_{[i]^k, [j]^k}=\M_{\{i,\ldots,i+k-1\},\{j,\ldots,j+k-1\}}.
		\]
		We say that $\M$ is an \textit{$SL_k$-tiling} if $M_{i\,j}\in SL_k(\Z)$ for all $i,j\in\Z$.
		In this case we say that $\M$ satisfies the \textit{$SL_k$-property}.
		We say that an $SL_k$-tiling $\M$ is \textit{tame} if every adjacent $(k+1)\times(k+1)$ sub-matrix of $\M$ has determinant $0$.
		We denote by $\slk$ the set of all tame $SL_k$-tilings.
	\end{definition}

	\begin{remark}
		Condition $(2)$ of Definition \ref{def:frieze} is not always required in defining friezes.
		It is also called \textit{tameness}.
		In this paper we only consider tame friezes and tilings, but \textit{wild} tilings, those not satisfying this condition, have been studied by Cuntz \cite{c15}.
	\end{remark}
	
	Friezes correspond to a special type of periodic tilings described as follows.
	
	\begin{definition} \label{def:totiling}\cite[p. 266]{br10}
		Let $F$ be an $SL_k$-frieze.
		Then we denote by $\M_F=(a_{ij})_{i,j\in\Z}$ the tiling constructed from $F$ using the following process.
		Let the rows of a frieze become falling diagonals of a tiling by rotating the frieze $45^\circ$ clockwise  and then reflected across a vertical line. 
		We may fill in the rest of the tiling through a skew extension in both directions as follows.  
		The entries $a_{i\,j}$ of the tiling $\M_F$ are defined by the equation $a_{i\,j+w+k+1}=(-1)^{k-1}a_{i\,j}$ where $w$ is the width of the frieze.  See Figure \ref{friezetotiling} for an example.
	\end{definition}
		
	\begin{figure}[H]
		\begin{center}
			\[
			\begin{array}{rrrrrrrrrrr}
			&\vdots&\vdots&\vdots&\vdots&\vdots&\vdots&\vdots&\vdots&\vdots\\
			\cdots&0&1&a&b&1&0&-1&-a&-b&\cdots\\
			
   \cdots&-1&0&1&c&d&1&0&-1&-c&\cdots\\
			\cdots&-f&-1&0&1&e&f&1&0&-1&\cdots\\
			\cdots&-g&-h&-1&0&1&g&h&1&0&\cdots\\
			&\vdots&\vdots&\vdots&\vdots&\vdots&\vdots&\vdots&\vdots&\vdots
			\end{array}
			\]
			\caption{Tiling resulting from rotating and extending the frieze in Figure \ref{friezeexample}.}
			\label{friezetotiling}
		\end{center}
		
	\end{figure}
	
	Another important structure in our study of tilings is $k$-column vectors.
	Our goal is to construct a bijection between $SL_k$-tilings and pairs of sequences of vectors satisfying certain properties.
	
	\begin{definition}\label{def:path}
		Let $\gamma=\{\gamma_i\}_{i\in\Z}$ be a bi-infinite strip of $k$-column vectors $\gamma_i\in\Z^k$ with the property that for every $i$ the matrix $\left(\gamma_i,\ldots,\gamma_{i+k-1} \right)$, whose columns are $k$ consecutive entries of $\gamma$, is an element of $SL_k(\Z)$.
		We denote the set of all such strips $\P_k$ and we call $\gamma$ a \textit{path}.
	\end{definition}

	These paths also contain the information about how any adjacent $k+1$ vectors are related.  	We encode this information in terms of $J$ matrices, which record how a column of $\gamma$ relates to the preceding $k$ columns.
	
	\begin{definition}\label{def:jmatrices}
		We define a matrix $J_n \in SL_k(\Z)$ for $n\in \Z$ as follows
		\[
		J_n=
		\left[
		\begin{array}{c|c}
		0\cdots0 & (-1)^{k-1}  \\
		\hline
		I_{k-1} & \begin{array}{c}
		j_{n\,2}\\
		\vdots\\
		j_{n\,k}
		\end{array} 
		\end{array}
		\right]
		\]
		where $I_{k-1}$ is the identity matrix of size $k-1$.
        We may also refer to the entry in the top right corner as $j_{n\,1}$ for consistency of notation.
		We use the term \textit{$J$ matrices} when referring to matrices of this form in general, and we write subscript $n$ as we will often refer to sequences of $J$ matrices indexed by $n$.
	\end{definition}

	Observe that multiplying a matrix $A=(a_1,\ldots,a_k)$ on the right by a matrix $J_n$ results in applying the following steps to $A$.
	\begin{enumerate}
		\item Delete the first column $a_1$ of $A$.
		\item Shift all remaining columns to the left by one.
		\item Append a new final column $a_1'$ which is a linear combination of the columns of $A$, namely
		\[
		a_1'=(-1)^{k-1}a_1+j_{n\,2}a_2+\cdots+j_{n\,k}a_k.
		\]
	\end{enumerate}
	The result is the matrix
	\[
	AJ_n=(a_2,\ldots,a_k,a_1').
	\]

	Thus, if $A\in SL_k(\Z)$, then so is $AJ_n$.
	In fact, for any matrix $A'\in SL_k(\Z)$ where the first $k-1$ columns of $A'$ are $a_2,\ldots, a_k$, there exists a unique $J_n$ such that $AJ_n=A'$.
	
	We later show that for any general $B\in SL_k(\Z)$, there is a sequence of $J$ matrices which allow us to transform $A$ into $B$ via right multiplication.
	Thus, these $J$ matrices serve as transitions between $k$ consecutive columns of a path.
	More precisely, to a path $\gamma\in\P_k$, we associate a collection of $J$ matrices $\{J_i\}_{i\in\Z}$ called \textit{transition matrices} such that 
	\[
	(\gamma_i,\ldots,\gamma_{i+k-1})J_i=(\gamma_{i+1},\ldots,\gamma_{i+k}).
	\]
	
	On the other hand, Bergeron and Reutenauer show that all rows (respectively columns) of a tiling $\M\in\slk$ may be written as a linear combination of the previous $k$ rows (respectively columns) with a coefficient of $(-1)^{k-1}$ on the first \cite[Lemma 2]{br10}.
    In particular, they state that
    \[
    \mathrm{Row}_{i+k}=(-1)^{k-1}\mathrm{Row}_i+j_{i\,2}\mathrm{Row}_{i+1}+\cdots+j_{i\,k}\mathrm{Row}_{i+k-1}\]\[
    \mathrm{Col}_{i+k}=(-1)^{k-1}\mathrm{Col}_i+j'_{i\,2}\mathrm{Col}_{i+1}+\cdots+j'_{i\,k}\mathrm{Col}_{i+k-1},
    \]
    where $\mathrm{Row}_i$ (respectively $\mathrm{Col}_i$) refers to the $i$-th row (respectively column) of $\M$.
	They refer to the coefficients, $\{j_{i\,2},j_{i\,3},\ldots,j_{i\,k}\}$ and $\{j'_{i\,2},j'_{i\,3},\ldots,j'_{i\,k}\}$, as the linearization data, and these, in turn, correspond directly with the last column of the $J$ matrix. 	They prove the following result.
	
	\begin{proposition}\label{prop:linearizationdata}\cite[Proposition 3]{br10}
		The mapping
		\begin{align*}
		\xi:\slk&\rightarrow SL_k(\Z)\times \left(\Z^{1\times (k-1)}\right)^\Z\times\left(\Z^{1\times (k-1)} \right)^\Z\\ 
		\M&\mapsto (M_{1\,1},\lambda,\mu),
		\end{align*}
		which associates to a tame $SL_k$-tiling its linearization data $\lambda$ and $\mu$ of the rows and columns, respectively, together with a central matrix $M_{1\,1}=\M_{[k],[k]}$, is a bijection.
	\end{proposition}

	This demonstrates that $J$ matrices also serve as transitions between rows and columns of $SL_k$-tilings.
		We recall a few additional definitions appearing in \cite{br10}, starting with the notion of periodicity.

	\begin{definition}\label{def:periodic}
		Let $m\in\Z_{>0}$.  
		We say a path $\gamma$ is \textit{$p$-periodic} if it has the property that $\gamma_i=\gamma_{i+p}$ for all $i\in\Z$.
		Similarly, a sequence of $J$ matrices $\{J_i\}_{i\in\Z}$ is \textit{$p$-periodic} if they have the property that $J_i=J_{i+p}$ for all $i\in\Z$. 
		We say a tiling $\M$ is \textit{$p$-row periodic} (respectively \textit{$p$-column periodic}) if $m_{i\,j}=m_{i+p\,j}$ (respectively $m_{i\,j}=m_{i,j+p}$) for all $i,j\in\Z$.
		We say that a tiling $\M$ is \textit{$(p\times q)$-periodic} if
		\[
		m_{i\,j}=m_{i+p\,j}=m_{i\,j+q}= m_{i+p\,j+q}
		\]
		for all $i,j\in\Z$.
	\end{definition}

We can similarly define the notion of skew periodic paths and tilings as follows. 	
	
	\begin{definition}\label{def:skewperiodic}
		Let $m\in\Z_{>0}$.  
		We say a path $\gamma$ is \textit{skew $p$-periodic} if it has the property that $\gamma_i=(-1)^{k-1}\gamma_{i+p}$ for all $i\in\Z$.
		We denote by $\P_{k,p}$ the set of all skew $p$-periodic paths. 
		We say a tiling $\M$ is \textit{skew $p$-row periodic} (respectively \textit{skew $p$-column periodic}) if $m_{i\,j}=(-1)^{k-1}m_{i+p\,j}$ (respectively $m_{i\,j}=(-1)^{k-1}m_{i,j+p}$) for all $i,j\in\Z$.
		We say that a tiling $\M$ is \textit{skew $(p\times q)$-periodic} if
		\[
		m_{i\,j}=(-1)^{k-1}m_{i+p\,j}=(-1)^{k-1}m_{i\,j+q}= m_{i+p\,j+q}
		\]
		for all $i,j\in\Z$.
	\end{definition}

	Furthermore, \cite{br10} introduces an interesting operation on tilings called duality which we recall below.
	Later, we will present an alternative interpretation of this by applying our results.
	
	\begin{definition}\label{def:dual}\cite[Equation 10]{br10}
		The \textit{$p$-derived tiling} of a tiling $\M$, denoted $\partial_p\M$ is given by
		\[
		\partial_p\M:=\left(M_{i\,j}^{(p)} \right)_{i,j\in\Z}
		\]
		where $M_{i\,j}^{(p)}$ is the adjacent $p\times p$ minor of $\M$ with upper-left entry $m_{i\,j}$, that is $M_{i\,j}^{(p)}=\det \M_{[i]^p, [j]^p}$.
		When $p=k-1$, we call $\partial_{k-1}\M$ the \textit{dual} of $\M$ and we write $\M^*$.
	\end{definition}

	The following proposition shows that the terminology of the dual is justified.
	
	\begin{proposition}\cite[Proposition 6]{br10}
		The dual of a tame $SL_k$-tiling is a tame $SL_k$-tiling.  Moreover, $(\M^*)^*$ and $\M$ coincide up to translation. 
	\end{proposition}

	\subsection{$SL_2$-tilings}
	
	$SL_2$-tilings were studied in detail by Short where he relates them to the combinatorics of the Farey graph \cite{s22}.
	We recall the main results below.
	
	\begin{definition}
		A \textit{Farey graph} is a graph with vertices in $\Q\cup\infty$.
		Two reduced rationals $\dfrac{a}{b}$ and $\dfrac{b}{c}$ are connected by an edge if $ad-bc=\pm1$.
		In this case, we take $\infty=\dfrac{1}{0}$.
	\end{definition}

	\begin{definition}
		A \textit{path of reduced rationals} is a bi-infinite sequence of reduced rationals $\gamma=\left\{\dfrac{a_i}{b_i}\right\}_{i\in\Z}$ which satisfy the property that
		\begin{equation}\label{eq1}
		a_ib_{i+1}-a_{i+1}b_i=1.
		\end{equation}
		We denote the set of all such paths as $\P\Q$.
		Note that elements of $\P\Q$ are paths in the graph theoretic sense in the Farey graph.
	\end{definition}
	
	We take a moment to draw an important distinction between $\P\Q$ and $\P_2$ as defined in Definition \ref{def:path}.
	Observe that for a path $\gamma\in\P_2$, there is a difference between columns $\displaystyle\binom{a}{b}$ and $\displaystyle\binom{-a}{-b}$ of $\gamma$ which is lacking in the Farey graph.
	Both of these represent the same reduced rational $\dfrac{a}{b}=\dfrac{-a}{-b}$. 
	Nevertheless, the condition \eqref{eq1} ensures that the same path in the Farey graph can be represented in two ways as $\gamma$ and $-\gamma$, obtained by negating numerator and denominator of every entry of $\gamma$. 
		This causes Short to identify tilings $\M$ and $-\M$, the tiling obtained by negating all entries of $\M$.
	This enables him to obtain a bijection between $SL_2(\Z)$-tilings up to a global change of sign and paths in $\P\Q$, which can be stated as follows. 

	Let $(\P\Q\times\P\Q)/SL_2(\Z)$ denote the set of conjugacy classes consisting of pairs of paths $\gamma,\delta$ up to multiplying all columns in $\gamma$ and $\delta$, now treated as vectors in $\Z^2$ rather then fractions, by the same matrix $A\in SL_k(\Z)$.	
	\begin{definition}\label{def:shortmap}\cite{s22}
		Let $\gamma=\left\{\disp\frac{a_i}{b_i}\right\}_{i\in\Z}$ and $\delta=\left\{\disp\frac{c_j}{d_j}\right\}_{j\in\Z}$ be paths in $\P\Q$.
		Define the map $\Phi\Q$ as follows.
		\begin{align*}
		\Phi\Q:(\P\Q\times\P\Q)/SL_2(\Z)&\rightarrow \mathbb{SL}_2/\pm\\
		\left(\gamma,\delta\right)&\mapsto \M=\left(m_{i\,j}\right)_{i,j\in\Z}
		\end{align*}
		where $\mathbb{SL}_2/\pm$ is the set of all tilings where we identify $\M$ and $-\M$, and $m_{i\,j}=a_id_j-b_ic_j$.
	\end{definition}

	\begin{theorem}\cite[Theorem 1.1]{s22}\label{thm:short}
		The map $\Phi\Q$ in Definition \ref{def:shortmap} is a bijection between pairs of paths of reduced rationals modulo $SL_2(\Z)$ and the set $\mathbb{SL}_2/\pm$ of all $SL_2$-tilings  up to a global change of sign.
	\end{theorem}
	
	In this paper, we obtain a generalization of this result for all $k\geq 2$.
	Moreover, the geometry of the Farey graph allows to establish further bijections by placing certain restrictions on paths.  In particular, this leads to the classification of friezes as well as positivity for friezes and tiling in terms of paths. 
	While the main result of \cite{s22} and several others are extended here, it is not clear how to generalize all of them as we lack the connection to geometry.
	
	\section{The Bijection}\label{sec:bijection}

	In this section we define a map $\Phi$ from pairs of paths to $SL_k$-tilings and then prove that it is a bijection.  We first establish a relationship between the $J$ matrices of paths and the $J$ matrices of tilings, which play an important part in the proof.  Furthermore, we rely on the interpretation of $SL_k$-friezes as Pl\"ucker friezes evaluated at a particular element of the Grassmannian.  
 	
	\subsection{Defining $\Phi$}
 
	We begin by introdusing a map $\widetilde{\Phi}$.
	The goal of this subsection is to show that this map is well-defined.
	
	\begin{definition}\label{phihat}
		We define a map $\widetilde{\Phi}$ from a pair of paths $(\gamma,\delta)\in\P_k\times\P_k$ to a tiling as follows  
		\begin{align*}
		\widetilde{\Phi}:(\P_k\times\P_k)&\rightarrow\slk\\
		(\gamma,\delta)&\mapsto \M=(m_{i\,j})_{i,j\in\Z},
		\end{align*} 
		where
		\[
		m_{i\,j}=\det\left(\gamma_i,\gamma_{i+1},\ldots,\gamma_{i+k-2},\delta_{j} \right).
		\]
	\end{definition}

    Thus, the entries $m_{i\,j}$ of the tiling come from taking the determinant of $k-1$ consecutive columns of $\gamma$ and a single column of $\delta$.

    \begin{example}\label{ex:phimapexample}
    Suppose we have a pair of paths $\gamma,\delta\in\Z^3$ defined as follows:
    \[  
    \gamma=(\cdots,\gamma_1,\gamma_2,\gamma_3,\gamma_4,\cdots)=\left(
    \begin{array}{cccccc}
        &1&0&0&1&\\
        \cdots&0&1&0&5&\cdots\\
        &0&0&1&2&
    \end{array}
    \right)\\
    \]
    and 
    \[
    \delta=(\cdots,\delta_1,\delta_2,\delta_3,\cdots)=\left(
    \begin{array}{ccccc}
        &1&1&1&\\
        \cdots&1&2&3&\cdots\\
        &1&3&6&
    \end{array}
    \right).
    \]
    
    To derive a specific entry of the tiling $\tilde{\Phi}$, say $m_{1\,2}$, we take two entries of $\gamma$ starting at $\gamma_1$ and append $\delta_2$.
    The resulting determinant gives the entry:
    \[
    m_{1\,2}=\det(\gamma_1,\gamma_2,\delta_2)=\det\begin{pmatrix}
        1&0&1\\
        0&1&2\\
        0&0&3
    \end{pmatrix}=3.\]
    We may assemble the whole tiling in this fashion.
    For example,
    \[
    \M=
    \begin{array}{rrrrr}
         \ddots&\vdots&\vdots&\vdots&\\
         \cdots&~~1&~~\textcolor{red}{3}&~~6&\cdots\\
         \cdots&~~1&~~1&~~1&\cdots\\
         \cdots&-4&-3&-2&\cdots\\
         &\vdots&\vdots&\vdots&\ddots
    \end{array}.
    \]
    \end{example}
    
	Note that it is not clear from the definition that the array $\widetilde{\Phi}(\gamma,\delta)$ satisfies the $SL_k$-property.
	We will show this later.
	For now, we need some preliminary results.
	Let $A$ be a $k\times k$ matrix.
	We define the multiplication of $A$ by $\gamma\in\P_k$, denoted by $A\gamma$, by
	\[
	A\gamma=\left(\cdots,A\gamma_1,A\gamma_2,\ldots \right).
	\]
	
	\begin{lemma}\label{lem:invariance}
		The map $\widetilde{\Phi}$ is invariant under multiplication by $A\in SL_k(\Z)$, that is $\widetilde{\Phi}(\gamma,\delta)=\widetilde{\Phi}(A\gamma,A\delta)$.
	\end{lemma}

	\begin{proof}
		Fix $(\gamma,\delta)\in(\P_k\times\P_k)$.
		Let $\M=(m_{i\,j})_{i,j\in\Z}$ be the image of $(\gamma,\delta)$ under $\widetilde{\Phi}$.
		Then $\widetilde{\Phi}(A\gamma,A\delta)$ equals $\M'=(m'_{i\,j})_{i,j\in\Z}$ where
		\begin{align*}
		m'_{i\,j}&=\det\left(A\gamma_i,\ldots,A\gamma_{i+k-2},A\delta_j\right)\\
		&=\det\left(A\left(\gamma_i,\gamma_{i+1},\dots,\gamma_{i+k-2},\delta_{j} \right)\right)\\
		&= \det(A)\det(\gamma_i, \gamma_{i+1}, \dots, \delta_j)\\
		&=\det(A)m_{i\,j}\\
		&=m_{i\,j}.
		\end{align*}
	\end{proof}
	
	Let $\pi:(\P_k\times \P_k)\rightarrow (\P_k\times \P_k)/SL_k(\Z)$ be the quotient map.  
	By Lemma \ref{lem:invariance}, $\widetilde{\Phi}(A\gamma,A\delta)=\widetilde{\Phi}(\gamma,\delta)$, so $\widetilde{\Phi}$ induces a map $\Phi$ from $(\P_k\times \P_k)/SL_k(\Z)$ to $\slk$ such that the following diagram commutes.
	\[
\begin{tikzcd}
	\P_k\times\P_k \ar[d,"\pi"] \ar[r,"\widetilde{\Phi}"] & \slk \\
	(\P_k\times \P_k)/SL_k(\Z) \ar[ur, "\Phi"]
	\end{tikzcd}
	\]
	We now focus on the map $\Phi$.  We begin by showing that $\Phi$ is well defined in the case of tilings obtained from friezes.  
	
	\begin{definition}\label{def:matrixpath}
	Let $A=(a_1, \dots, a_n)$ be a $k\times n$ matrix with entries in $\Z$ such $p_{o(I)}(A)=1$ for all consecutive Pl\"ucker coordinates. Define a path $\phi_A \in \P_k$ by taking the skew periodic extension of $A$ by letting 
	$(\phi_A)_i=a_i$ for  $i\in[n]$ and otherwise $(\phi_A)_i=(-1)^{k-1} (\phi_A)_{i+n}$. 
	\end{definition}

	\begin{lemma}\label{lem:frieze}
	Let $A=(a_1, \dots, a_n)$ be a $k\times n$ matrix with entries in $\Z$ such $p_{o(I)}(A)=1$ for all consecutive Pl\"ucker coordinates.  Then the tiling from the frieze $\F_{(k,n)}(A)$ equals the image of $(\phi_A, \phi_A)$ under $\Phi$, that is $\M_{\F_{(k,n)}(A)} = \Phi(\phi_A, \phi_A)$.
	\end{lemma}
	
	\begin{proof}
	The entries of the frieze $\F_{(k,n)}(A)$, and hence the tiling $\M_{\F_{(k,n)}(A)}$, correspond to almost consecutive minors $p_{o(I)}(A)$ up to a sign.  Note that by definition of $\M_{\F_{(k,n)}(A)}$, the rows of $\F_{(k,n)}(A)$ become falling diagonals of the tiling. We can index the entries of the tiling such that the entry in position $(1,1)$ equals $p_{o([1]^{k-1},1)}(A)=0$, which lies on the first diagonal of zeros.  More generally, the entry in position $(i,j)$ equals $p_{o([i]^{k-1},j)}(A)$ up to a sign of $(-1)^k$.  Similarly, by definition of the map $\Phi$, the entry in position $(i,j)$ equals 
	\[\det((\phi_A))_i, \dots, (\phi_A)_{i+k-2}, (\phi_A))_j)=\pm p_{o([i]^{k-1},j)}(A)\]
	where equality follows from the construction of $\phi_A$.  
	Furthermore, it is easy to see that in the diagonal strip where $i\leq j \leq n+i-1$, the strip obtained by rotating the frieze $\F_{(k,n)}(A)$, we have that the two quantities above have the same sign, and therefore are actually equal.  Then both $\M_{\F_{(k,n)}(A)}$ and  $\Phi(\phi_A, \phi_A)$ are obtained by extending this strip to the left and to the right skew periodically, which shows the desired result.
				\end{proof}
	
	Next, we make an observation about $J$ matrices.
	
	\begin{lemma}\label{generators}
		The $J$ matrices generate $SL_k(\Z)$.
	\end{lemma}
	
	\begin{proof}
		It is a classical result in linear algebra that shear matrices generate the special linear group, see for example \cite{a57}.
		Thus, it suffices to show that $J$ matrices generate shear matrices $S_{i\,j}(\lambda)$ which consist of the identity matrix with a single nonzero entry $\lambda$ in position $(i\,j)$ where $i\neq j$.
		
		Fix a shear matrix $S_{i\,j}(\lambda)$.
		Recall the structure of a $J$ matrix from Definition \ref{def:jmatrices}.
		Consider the power of a matrix $(J_n)^{j-1}$ where $J_n$ is the $J$ matrix with $j_{n\,\ell}=0$ for all $\ell\in\{2,\ldots,k\}$.
		This matrix $(J_n)^{j-1}$ has columns
		\[
		(e_j,\ldots,e_k,(-1)^{k-1}e_1,\ldots,(-1)^{k-1}e_{j-1}).
		\]
		Then, multiply by the $J$ matrix $J_m$ with a single nontrivial $j$-entry $j_{m,i'}$ in the last column such that  
		\[
		j_{m,i'}=\begin{cases}
		\lambda&\text{if }i>j\\
		(-1)^{k-1}\lambda&\text{if }i<j\\
		\end{cases}
		\]
		where $i'\equiv (i-j+1)\pmod{k}$ and $j_{m\,\ell}=0$ for $\ell\not=i'$.
		The resulting matrix $(J_n)^{j-1}J_m$ is of the form 
		\[
		(J_n)^{j-1}J_m=\left(e_{j+1},\ldots,e_k,(-1)^{k-1}e_1,\ldots,(-1)^{k-1}e_{j-1},(-1)^{k-1}(\lambda e_i+e_j)\right).
		\]
		Then multiply the above by $(J_n)^{k-j}$ on the right to obtain $(-1)^{k-1}S_{i\,j}(\lambda)$.
		Finally, to get the desired $S_{i\,j}(\lambda)$, multiply by $(J_n)^k$.
	\end{proof}
	
	From this, we develop the following technique of constructing a single path $\phi(\gamma,\delta)$ from a pair of paths $\gamma, \delta \in P_k$. 
	
		\begin{definition}\label{def:friezemap}
	Consider a pair of paths $\gamma, \delta\in P_k$ and a pair of integers $m,n\geq k$.  
	We define a periodic path with $m$ consecutive entries from $\gamma$ and $n$ consecutive entries from $\delta$ as follows. 
	Without loss of generality, we may select a labeling such that the desired elements of $\gamma$ and $\delta$ are $\{\gamma_1,\ldots,\gamma_m\}$ and $\{\delta_1,\ldots,\delta_n\}$ respectively.
	By Lemma \ref{generators}, there is a sequence of $J$ matrices $\{J_i\}_{i\in[\ell+k]}$ for some $\ell\in\Z$ such that
	\[
	(\gamma_{m-k+1},\ldots,\gamma_{m})J_1\cdots J_{\ell+k}=(\delta_1,\ldots, \delta_k).
	\]
	Let $\lambda_i$ with $i\in[\ell]$ be the last column of the product
	\[
	(\gamma_{m-k+1},\ldots,\gamma_{m})J_1\cdots J_i.
	\]
	Similarly, there is a sequence of $J$ matrices $\{J'_i\}_{i\in [q+k]}$ for some $q\in\Z$ such that
	\[
	(\delta_{n-k+1},\ldots,\delta_{n})J_1'\cdots J'_{q+k}=(-1)^{k-1}(\gamma_1,\ldots, \gamma_k).
	\]
	Let $\mu_i$ with $i\in[q]$ be the last column of the product
	\[
	(\delta_{n-k+1},\ldots,\delta_{n})J_1'\cdots J_i'.
	\]
	We obtain a matrix $A$ with columns
	\begin{equation}\label{eqn:deltaprime}
	A=(\gamma_1,\dots,\gamma_m,\lambda_1,\ldots,\lambda_{\ell},\delta_1,\ldots,\delta_n,\mu_1,\ldots,\mu_{q}).
	\end{equation}
	Note that, by construction, $A$ satisfies the property that all consecutive Pl\"ucker coordinates of $A$ are $1$.  Then we define a path 
	\[\phi_{m,n}(\gamma,\delta)= \phi_A \in \P_k\] 
	to be the skew periodic extension of $A$.
When obvious, we will omit the subscripts $m,n$ and write $\phi(\gamma,\delta)$.	
	\end{definition}

    Note that the matrix $A$ above depends on the particular choice of $J$ matrices needed to get from columns of $\gamma$ to $\delta$ and back.  We provide an example of the construction of a path $\phi(\gamma,\delta)$ below.

    \begin{example}\label{ex:friezemapping}
        Suppose we have a pair of paths $\gamma,\delta\in\Z^2$ defined as follows
        \[
        \gamma=\left(\cdots
        \begin{array}{rrrr}
           0&-1&-2&-1\\
           1&1&1&0
        \end{array}
        \cdots
        \right)\quad\text{and}\quad\delta=\left(
        \cdots\begin{array}{rrr}
             -4&1&2\\
             3&-1&-1
        \end{array}\cdots
        \right).
        \]
        We wish to construct a path $\phi_{3,2}(\gamma,\delta)$.  
        We start with the first three columns of $\gamma$
        \[
        \left(
        \begin{array}{rrr}
           0&-1&-2\\
           1&1&1
        \end{array}
        \right).
        \]
        Our goal is to use this starting place to create a path which reaches the first two columns of $\delta$, i.e.
        \[
        \left(
        \begin{array}{rrrrrr}
             0&-1&-2&\cdots&-4&1\\
             1&1&1&\cdots&3&-1
        \end{array}
        \right).
        \]
        We may do this by constructing an appropriate sequence of $J$ matrices which generate
        \[
        \left(
        \begin{array}{rr}
            -1&-2\\
            1&1
        \end{array}\right)^{-1}\left(\begin{array}{rr}
            -4&1\\
            3&-1
        \end{array}\right) = \left(\begin{array}{rr}
            2&-1\\
            1&0
        \end{array}\right).
        \]
        For example, we can express this matrix as a product of $J$ matrices as follows. 
        \[
       \left(\begin{array}{rr}
            2&-1\\
            1&0
        \end{array}\right)  = \left(\begin{array}{rr}
            0&-1\\
            1&-1
        \end{array}\right)
        \left(\begin{array}{rr}
            0&-1\\
            1&-2
        \end{array}\right)
        \left(\begin{array}{rr}
            0&-1\\
            1&-1
        \end{array}\right) =  J_1 J_2 J_3
        \]
        Thus, we obtain a single vector $\lambda_1=\left(\begin{array}{r}
            3\\
            -2
        \end{array}\right)$.
        We do the same in order to get from $(\delta_1, \delta_2)=\left(\begin{array}{rr}
            -4 & 1\\
            3 & -1
        \end{array}\right)$ back to the initial entry of $\gamma$ scaled by $-1$.
        A similar computation, yields a sequence of vectors $\mu$ consisting of a single vector $\mu_1=\left(\begin{array}{r}
        -1\\ 
        2\end{array}\right)$.
        The resulting path has the form
        \begin{align*}
        \phi_{3,2}(\gamma,\delta)&=(\cdots, \gamma_1,\gamma_2,\gamma_3,\lambda_1,\delta_1,\delta_2,\mu_1,-\gamma_1,-\gamma_2,-\gamma_3,\cdots)\\
        &=\left(\cdots
        \begin{array}{rrrrrrrrrr}
             0&-1&-2&3&-4&1&-1&0&1&2\\
             1&1&1&-2&3&-1&2&-1&-1&-1
        \end{array}
        \cdots
        \right).
        \end{align*}
    \end{example}

	From this partial data about $\gamma$ and $\delta$, we will be able to extract certain partial information about their image $\M=\Phi(\gamma,\delta)$ and then expand this to a variety of conclusions about the whole image $\M$.
	
	\begin{lemma}\label{lem:friezemapping}
		Let $\M:=\Phi(\gamma,\delta)$ be an array and let $M=\M_{\{1,\ldots,m\},\{1,\ldots,n\}}$ be an adjacent $m\times n$ sub-matrix in $\M$ with $m,n\geq k$.
		Then a tame $SL_k$-tiling $\M'=\Phi(\phi_{m,n}(\gamma,\delta), \phi_{m,n}(\gamma,\delta))$ from a frieze  has the property that the adjacent $m\times n$ sub-matrix
		\[M'=\M'_{\{1,\ldots,m\},\{i,\ldots,i+n-1\}}
		\] 
		is equal to $M$ for some $i$.
	\end{lemma}
	
	\begin{proof}
		For paths $\gamma$ and $\delta$ use equation (\ref{eqn:deltaprime}) to construct a matrix $A$ of size $k\times (m+\ell+n+q)$ and a path $\phi(\gamma,\delta)=\phi_{m+k-2,n}(\gamma,\delta)$ as in Definition \ref{def:friezemap}.  Note that by construction $\phi(\gamma,\delta)_{[m]}=\gamma_{[m]}$ and $\phi(\gamma,\delta)_{[m+\ell]^n}=\delta_{[n]}$.
		The matrix $A$ satisfies the condition that all consecutive Pl\"ucker coordinates equal 1, so Lemma~\ref{lem:frieze} shows that $\M' = M_{\F_{(k,m+\ell+n+q)}(A)}$ is a tame $SL_k$-tiling.
		
				Let $M'$ be a submatrix of $\M'$ of size $m\times n$ with upper-left entry $p_{1,\ldots,k-1,m+\ell+1}(A)$.
		By construction, this is given by the determinant of the matrix $(\gamma_1,\cdots,\gamma_{k-1},\delta_1)$, and is therefore equal to $m_{1,1}$ in $\M$.
		The same holds for all other entries of $M'$.  
	\end{proof}

	This finally allows us to show that $\Phi$, and hence $\widetilde{\Phi}$, are well-defined.
	
	\begin{proposition}\label{phi}
		The map $\Phi:(\P_k\times\P_k)/SL_k(\Z)\rightarrow\slk$ is well-defined.
	\end{proposition}
	
	\begin{proof}
		Let $(\gamma,\delta)\in(\P_k\times\P_k)/SL_k(\Z)$ and $\M=\Phi(\gamma,\delta)$.
        We wish to show that $\M\in\slk$.
		By Lemma \ref{lem:friezemapping}, we can map $M=\M_{\{1,\ldots,m\},\{1,\ldots,n\}}$ to a portion of a tame $SL_k$-tiling $\M'$.
		Observe that in equation (\ref{eqn:deltaprime}), our choice of indexing for $\gamma$ and $\delta$ was arbitrary.
		Therefore, we may adjust our labels on $\gamma$ and $\delta$ such that the upper left corner of $M$ may be any entry in $\M$.
		Therefore, since $\M'\in\slk$, we have that, in particular, every adjacent $k\times k$ sub-matrix is in $SL_k(\Z)$ and every adjacent $(k+1)\times (k+1)$ adjacent sub-matrix has determinant $0$.
		Thus, $\M$ is in $\slk$.
	\end{proof}

	\subsection{Transition Matrices}
	We begin by formally defining transition matrices for tilings. 
	Then we present a series of results which will allow us to compare transition matrices of tilings and paths.

    \begin{definition}\label{def:horizontalandverticaltransitions}
        Given an $SL_k$-tiling $\M$ we write $H_j$ for the \emph{ horizontal transition matrix} of $\M$ which transitions from column $j+k-1$ to $j+k$, i.e. for all $i$
        \[
        M_{i\,j}H_{j}=M_{i\,j+1}.
        \]
        Similarly, we write $V_i$ for the \emph{vertical transition matrix} which transitions from row $i+k-1$ to $i+k$, i.e. for all $j$
        \[
        M_{i\,j}^TV_{i}=M_{i+1\,j}^T.
        \]
    \end{definition}

Note that these transition matrices are in particular $J$ matrices. 
    With this notation in mind, we make the following observation.
	Let $\M=\M_F$ be a tiling from a frieze $F$ with entries in some integral domain.
	We may index $\M$ such that $M_{0\,0}:=M$ is of the following form.
	\[
	M=\begin{pmatrix}
	1&m_{1\,1}&m_{1\,2}&\cdots&m_{1\,(k-1)}\\
	0&1&m_{2\,2}&\cdots&m_{2\,(k-1)}\\
	&\ddots&\ddots&\cdots&\vdots\\
	&&0&1&m_{(k-1)\,(k-1)}\\
	\textbf{0}&&&0&1
	\end{pmatrix}.
	\]
	In order to transition from $M$ to the $k$-th column of $\M$, we use a horizontal transition matrix $H_0$ of $\M$ which takes the form of a $J$ matrix.
	In particular,
	\begin{equation}\label{eqn:friezetransition}
	M
		\begin{pmatrix}
		(-1)^{k-1}\\
		j_2\\
		\vdots\\
		j_{k-p+1}\\
		\vdots\\
		j_{k}
		\end{pmatrix}=
	\begin{pmatrix}
	m_{1\,k}\\
	m_{2\,k}\\
	\vdots\\
	m_{k-p+1\,k}\\
	\vdots\\
	m_{k\,k}
	\end{pmatrix},
	\end{equation}
	for $p\in[k]$ where $(
		(-1)^{k-1},
		j_2,
		\ldots,
		j_{k-p+1},
		\ldots,
		j_{k}
		)^T$
	is the final column of $H_0$.
    This gives rise to the following lemma.
	
	\begin{lemma}\label{lem:jmatrixminor}
		Let $\M=\M_F$ be a tiling from a frieze $F$ and let $(
			(-1)^{k-1},
			j_2,
			\cdots,
			j_{k-p+1},
			\cdots,
			j_{k}
			)^T$ be the final column of the horizontal transition matrix $H_0$.
		Then $j_{k-p+1}=(-1)^{p-1}M_p$ where $M_p$ is the $p$-th adjacent minor of $\M_{[k],[k]}$ whose bottom right corner is $m_{k\,k}$.
	\end{lemma}
	
	\begin{proof}
		We proceed by strong induction.
		Take as a base case $p=1$.
		By equation (\ref{eqn:friezetransition}) and the upper triangular structure of matrix $M$, clearly $j_{k}=m_{k\,k}$.
		Additionally, the $1$ minor is $|m_{k\,k}|=m_{k\,k}$.
		Thus, $j_{k}=(-1)^{1-1}m_{k\,k}=m_{k\,k}$, as desired, and the base case holds.
		
		Suppose the statement holds for $j_{k-i+1}$ where $i\in[p-1]$ with $p\leq k$.
		Consider the case $j_{k-p+1}$.
		By equation (\ref{eqn:friezetransition}), we see that
		\begin{equation}\label{eqn:jentries}
		j_{k-p+1}=m_{(k-p+1)\,k}-\sum_{i=1}^{p-1}j_{k-i+1}m_{(k-p+1)\,(k-i)}.
		\end{equation}
		On the other hand, the minor $M_p$ has the form
		\[
		M_p=\left|
		\begin{array}{ccccc}
		m_{(k-p+1)\,(k-p+1)}&m_{(k-p+1)\,(k-p+2)}&&\cdots&m_{(k-p+1)\,k}\\
		1&m_{(k-p+2)\,(k-p+2)}&&\cdots&m_{(k-p+2)\,k}\\
		&\ddots&&\cdots&\vdots\\
		&&1&m_{(k-1)\,(k-1)}&m_{(k-1)\,k}\\
		\textbf{0}&&&1&m_{k\,k}
		\end{array}
		\right|.
		\]
		We compute the determinant by going across the top row.
		In the expression for $M_p$, denote by $r_q$ the $(p-1)$-st minor associated with the term $m_{(k-p+1)\,(k-p+q)}$ where $q<p$ from the top row.
		This minor $r_q$ is of the form
		\[
		r_q=\left|
		\begin{array}{ccc|cccc}
		1&&*&\\
		&\ddots&&&*\\
		\textbf{0}&&1\\
		\hline
		&&&m_{(k-p+q+1)\,(k-p+q+1)}&\cdots&&m_{(k-p+q+1)\,k}\\
		&\textbf{0}&&1&\cdots&&m_{(k-p+q+2)\,k}\\
		&&&&\ddots&&\vdots\\
		&&&&&1&m_{k\,k}
		\end{array}
		\right|.
		\]
		This determinant is given entirely by the determinant of the lower-right block of size $p-q$, which, by the inductive hypothesis, is $(-1)^{p-q-1}j_{k-p+1+q}$.
		Thus, for $q<p$ each of the terms in the expression for $M_p$ are given by 
			\[
			(-1)^{q-1}m_{(k-p+1)\,(k-p+q)}r_q=(-1)^{q-1}m_{(k-p+1)\,(k-p+q)}(-1)^{p-q-1}j_{k-p+1+q}=(-1)^pj_{k-p+1+q}m_{(k-p+1)\,(k-p+q)}.
			\]
		When $q=p$, $r_p$ is given by $(-1)^{p-1}m_{(k-p+1)\,k}$ times the identity matrix.
		Thus, we have that the minor $M_p$ is given by
		\[
		M_p=(-1)^{p-1}m_{(k-p+1)\,k}+(-1)^p\sum_{i=1}^{p-1}j_{k-i+1}m_{(k-p+1)\,(k-i)}.
		\]
		Multiplying equation (\ref{eqn:jentries}) by $(-1)^{p-1}$ gives us that
		\[
		(-1)^{p-1}j_{k-p+1}=(-1)^{p-1}m_{(k-p+1)\,k}+(-1)^p\sum_{i=1}^{p-1}j_{k-i+1}m_{(k-p+1)\,(k-i)}=M_p,
		\]
		as desired.
	\end{proof}
	
	We generalize this result to give us the entries of all $J$ matrices in a Pl\"ucker frieze.
	We use the notation $j_{pq}$ for the entry in row $q$ in the last column of the horizontal transition matrix $H_p$.  We use analogous notation for vertical transition matrices.  
	
	\begin{proposition}\label{prop:jentries}
		
			Let $\M=\M_{\F(k,n)}$ such that $m_{1\,1}=p_{o([1]^{k-1},1)}$.
            \begin{enumerate}
                \item [(a)] The entry $j_{p\,q+1}$ of $H_p$ is given by
			\[
			j_{p\,q+1}=(-1)^{k-q-1}p_{o([p]^q,[p+q+1]^{k-q})}.
			\]
                \item [(b)]The entry $j_{p\,q+1}$ of $V_p$ is given by
			\[
			j_{p\,q+1}=(-1)^{k-q-1}p_{o([p+q-1]^{k-q},[p+k]^q)}.
			\]
            \end{enumerate}

	\end{proposition}
	
	\begin{proof}
            We begin by proving part (a).
			The horizontal transition matrix $H_{p+k-1}$ is a transition between the adjacent submatrix
            \[
            A_{[p+k-1]^k;p}=\begin{pmatrix}
                p_{o([p]^{k-1},p+k-1)} & p_{o([p]^{k-1},p+k)} & \cdots & p_{o([p]^{k-1},p+2k-2)}\\
                p_{o([p+1]^{k-1},p+k-1)} & p_{o([p+1]^{k-1},p+k)} & \cdots & p_{o([p+1]^{k-1},p+2k-2)}\\
                \vdots & \vdots & \ddots & \vdots\\
                p_{o([p+k-1]^{k-1},p+k-1)} & p_{o([p+k-1]^{k-1},p+k)} & \cdots & p_{o([p+k-1]^{k-1},p+2k-2)}
            \end{pmatrix}
            \]
             as in Definition \ref{def:pluckermatrix} and the next column of $\M$.
			By Lemma \ref{lem:jmatrixminor}, this means $j_{p+k-1\,q+1}$ is equal to $(-1)^{k-q-1}$ times the adjacent $(k-q)$-minor of $A_{[p+k]^{k};p}$ aligned with the bottom-right corner.  This corresponds to the matrix $A_{[p+k+q]^{k-q};p+q}$.  Note that it was not important in the proof of Lemma \ref{lem:jmatrixminor} that the entries $m_{ij}$ were integers, so we may apply it here where the entries are Pl\"ucker coordinates. 
			By Proposition \ref{prop:pluckerdeterminant}, we have that
			\[
			\det(A_{[p+k+q]^{k-q};p+q})=p_{o([(p+q)+(k-q)-1]^{k-(k-q)},[p+k+q]^{k-q})}=p_{o([p+k-1]^q,[p+k+q]^{k-q})}.
			\]
			Thus, we have
			\[
			j_{k+p-1\,q+1}=(-1)^{k-q-1}p_{o([p+k-1]^q,[p+k+q]^{k-q})}.
			\]
			Setting $p$ to $p-k+1$, we obtain the desired result.

            Part (b) follows similarly by using transposes of the appropriate matrices.
	\end{proof}

    We show an example of how the formulas in parts (a) and (b) of Proposition \ref{prop:jentries} are related.

    \begin{example}\label{ex:verticaljentries}
    We may first use the formula in Proposition \ref{prop:jentries} part (a) to construct the sequence of the final columns of the $J$ matrices of a Pl\"ucker frieze, coming from the horizontal transition matrices of the corresponding tiling.
    For example, if we take $k=4$, we have the following sequence starting with $p=1$:
    \[
    \left(\cdots,
    \left(
    \begin{array}{r}
    -1\\
    p_{1\,3\,4\,5}\\
    -p_{1\,2\,4\,5}\\
    \textcolor{blue}{p_{1\,2\,3\,5}}
    \end{array}
    \right)
    ,\left(
    \begin{array}{r}
    -1\\
    p_{2\,4\,5\,6}\\
    \textcolor{blue}{-p_{2\,3\,5\,6}}\\
    p_{2\,3\,4\,6}
    \end{array}
    \right)
    ,\left(
    \begin{array}{r}
    -1\\
    \textcolor{blue}{p_{3\,5\,6\,7}}\\
    -p_{3\,4\,6\,7}\\
    p_{3\,4\,5\,7}
    \end{array}
    \right)
    ,\left(
    \begin{array}{r}
    \textcolor{blue}{-1}\\
    p_{4\,6\,7\,8}\\
    -p_{4\,5\,7\,8}\\
    p_{4\,5\,6\,8}
    \end{array}
    \right)
    ,\cdots\right).
    \]
    We may similarly use the formula from Proposition \ref{prop:jentries} part (b) to construct the sequence for vertical transition matrices starting with $p=1$:
    \[
    \left(\cdots,\textcolor{blue}{\left(
    \begin{array}{r}
    -1\\
    p_{1\,2\,3\,5}\\
    -p_{2\,3\,5\,6}\\
    p_{3\,5\,6\,7}
    \end{array}
    \right)}
    ,\left(
    \begin{array}{r}
    -1\\
    p_{2\,3\,4\,6}\\
    -p_{3\,4\,6\,7}\\
    p_{4\,6\,7\,8}
    \end{array}
    \right)
    ,\left(
    \begin{array}{r}
    -1\\
    p_{3\,4\,5\,7}\\
    -p_{4\,5\,7\,8}\\
    p_{5\,7\,8\,9}
    \end{array}
    \right)
    ,\left(
    \begin{array}{r}
    -1\\
    p_{4\,5\,6\,8}\\
    -p_{5\,6\,8\,9}\\
    p_{6\,8\,9\,10}
    \end{array}
    \right)
    ,\cdots\right).
    \]
    Note that column entries along the ascending diagonals of one sequence form the columns of the other.
    \end{example}

    With this example as guidance, we define an operation and notation on paths which alters their $J$ matrices as above. 

    \begin{definition}\label{def:tildepaths}
        Given a path $\gamma\in \P_k$ with $J$ matrices $\{J_i\}_{i\in\Z}$, define a new path $\widetilde{\gamma}\in \P_k$ such that $(\widetilde{\gamma})_{[k]}=(\gamma_{[k]})^T$ and the $J$ matrices of $\widetilde{\gamma}$ written $\{\widetilde{J_i}\}_{i\in\Z}$ have entries in the final column given by
        \[
        \widetilde{j_{i\,q}}=\begin{cases}
            j_{i\,1}=(-1)^{k-1}&q=1\\
            (-1)^kj_{i+q-2,\,k-q+2}&\text{else}
        \end{cases}
        \]
        for $i\in\Z$ and $q\in[k]$.
    \end{definition}

    \begin{remark}\label{rem:tildeshift}
        Note that $\widetilde{\widetilde{J_i}}$, the $J$ matrices of $\widetilde{\widetilde{\gamma}}$, equal the $J$ matrices of $\gamma$, up to a shift in indexing by $k-2$, i.e.
        \[
        \widetilde{\widetilde{J_i}}=J_{i+k-2}
        \]
        for all $i\in\Z$.
    \end{remark}
    
    Next, we obtain results about the transition matrices of tilings.
    
    \begin{lemma}\label{lem:horizontaljmatrices}
        The horizontal transition matrices of the tiling $\M:=\Phi(\gamma,\delta)$ are equal to the $J$ matrices of $\delta$.
    \end{lemma}

    \begin{proof}
    Let $(\gamma,\delta)\in (\P_k\times\P_k)/SL_k(\Z)$.
    Without loss of generality, we may select $\gamma$ such that $\left(\gamma_1,\gamma_2,\ldots,\gamma_k \right)=I_k$.
    We want to show that the horizontal transition matrices of $\M$ match up with the $J$ matrices of $\delta$.
    It suffices to show that $M_{1\,1}J_1=M_{1\,2}$ where $J_1$ is the first transition matrix of $\delta$, i.e. $J_1=H_1$.
    In particular, we want to show 
    \begin{equation}\label{eq:72}
    \sum_{i=1}^kj_{1\,i}m_{\ell\,i}=m_{\ell\,k+1}
    \end{equation}
    for all $\ell\in[k]$, where $j_{1\,i}$ is the $i$-th entry in the last column of $J_1$.
    By definition of $\Phi$, we note that
    \begin{equation}\label{eq:71}
    m_{1\,i}=\det(\gamma_1,\dots,\gamma_{k-1},\delta_i)=\det(e_1, \dots, e_{k-1}, \delta_i)=\delta_{i\,k}.
    \end{equation} 
    Then we obtain 
    \[\sum_{i=1}^kj_{1\,i}m_{1\,i}=\sum_{i=1}^kj_{1\,i}\delta_{i\,k}=\delta_{k+1, k}=m_{1\,k+1}\]
    where the first and the last equalities follow from equation \eqref{eq:71}, and the middle equality follows from the definition of the $J$ matrix $J_1$ for the path $\delta$. In particular, this shows equation \eqref{eq:72} in the case $\ell=1$.
    
    For the remaining values of $\ell$, observe that the entry $m_{\ell+1\,i}$ is given by the determinant of the matrix
    \[
   B= \left(\begin{array}{c|c}
    \textbf{0}&\begin{array}{c|c}
    P&\begin{array}{c}
        \delta_{i\,1}\\
        \delta_{i\,2}\\
        \vdots\\
        \delta_{i\,\ell}
    \end{array}
    \end{array}\\
    \hline\\
    I_{k-\ell}&*
    \end{array}
    \right),
    \]
    where the $\ell\times(\ell-1)$ matrix $P$ consists of the upper $\ell$ entries of the matrix $\left(\gamma_{k+1},\gamma_{k+2},\ldots,\gamma_{k+\ell-1} \right)$.
    We may write $\det B$ as
    \begin{equation}\label{eqn:lincombodelta}
    m_{\ell+1\,i}=(-1)^{k-\ell}\sum_{n=1}^{\ell}a_n\delta_{i\,n},
    \end{equation}
    where $a_n$ is some fixed value based on $P$ used in taking the determinant of the upper right block of $B$.
    Note that $P$, and hence $a_n$, does not depend on $i$.
    Thus, we have
    \begin{align*}
        \sum_{i=1}^kj_{1\,i}m_{\ell+1\,i}&=(-1)^{k-\ell}\sum_{i=1}^k \Big(j_{1\,i}\sum_{n=1}^{\ell}a_n\delta_{i\,n}\Big)\\
        &=(-1)^{k-\ell}\Big(\sum_{n=1}^{\ell}a_n\Big)\Big(\sum_{i=1}^kj_{1\,i}\delta_{i\,n}\Big).
    \end{align*}
    By definition of $J_1$, we may rewrite the final sum as $\delta_{k+1\,n}$.
    Thus, we have
    \[
    \sum_{i=1}^kj_{1\,i}m_{\ell+1\,i}=(-1)^{k-\ell}\sum_{n=1}^{\ell}a_n\delta_{k+1\,n}=m_{\ell+1\,k+1},
    \]
    for all $\ell\in[k-1]$ where the last equality follows from equation (\ref{eqn:lincombodelta}).  This shows that equation \eqref{eq:72} holds.
    \end{proof}

    \begin{lemma}\label{lem:verticaljmatrices}
        The vertical transition matrices of the tiling $\M=\Phi(\gamma,\delta)$ are equal to the $J$ matrices of $\widetilde{\gamma}$. 
    \end{lemma}

    \begin{proof}
By Lemma \ref{lem:friezemapping}, we may construct an $SL_k$-tiling $\M'=\Phi(\phi(\gamma,\delta), \phi(\gamma,\delta))$ from a single path $\phi(\gamma,\delta)$ with contains an adjacent $m\times n$ matrix $\M'_{\{1,\ldots,m\},\{i,\ldots,i+n-1\}}$ equal to $\M_{[m],[n]}$ where $m,n>k$.
		By construction, the path $\phi(\gamma, \delta)$ has as its first $J$ matrix $J_1$ the first $J$ matrix of $\gamma$.
		Since $\M'=\Phi(\phi(\gamma,\delta), \phi(\gamma,\delta))$, then Lemma \ref{lem:horizontaljmatrices} implies that the horizontal transition matrices starting at $\M'_{[k],[k]}=M'_{1\,1}$ are the $J$ matrices of $\phi(\gamma,\delta)$, which are $J$ matrices of $\gamma$.
        Note that the matrix $\M'_{[k],[k]}$ need not lie in $\M$. 
		Since $\M'$ is a tiling from a frieze, by Proposition \ref{prop:jentries} part (b), the vertical transition matrix $V_1'$ has entries given by 
        \[
        j_{1\,q+1}=(-1)^{k-q-1}p_{o([q]^{k-q},[k+1]^q)}(A)
        \] 
        where$A$ is one period of $\phi(\gamma,\delta)$.
		Consider the entries of the matrix $\widetilde{J_1}$ of $\widetilde{\gamma}$.
        By Definition \ref{def:tildepaths} and Proposition \ref{prop:jentries}, these correspond to the entries of the final column of the $J$ matrix given by
        \[
        \widetilde{j_{1\,q+1}}=(-1)^kj_{q\,k-q+1}=(-1)^k (-1)^{k-(k-q+1)}p_{o([q]^{k-q},[k+1]^q)}=(-1)^{k-q-1}p_{o([q]^{k-q}[k+1]^q}),
        \]
        applied to $A$ for $q\in[k-1]$, where the second equality follows from Proposition~\ref{prop:jentries}(a).
        Thus, by taking $m,n$ large enough, the vertical transition matrix $V_1'$ equals the first $J$ matrix $\widetilde{J_1}$ of $\widetilde{\gamma}$.
        Now, we need to show that $V_1$, the vertical transition matrix of $\M$, equals $V_1'$.
        This follows by construction, since $M_{1\,1}=M'_{1\,1}$ and $M_{2\,1}=M'_{2\,1}$.
        Thus, $\widetilde{J_1}=V_1$, as desired.
		Since our choice of starting position for $\M'$ was arbitrary, this holds for all vertical transition matrices and all $J$ matrices of $\gamma$.
    \end{proof}

    We combine the above lemmas to obtain the following crucial result.
    
	\begin{proposition}\label{cor:jmatrices}
		Let $\M=\Phi(\gamma,\delta)$ be an $SL_k$-tiling.
        The horizontal transition matrices of $\M$ equal the $J$ matrices of $\delta$ and the vertical transition matrices of $\M$ equal the $J$ matrices of $\widetilde{\gamma}$.
	\end{proposition}
	
	\begin{proof}
		This result follows from Lemmas \ref{lem:horizontaljmatrices} and \ref{lem:verticaljmatrices}.
	\end{proof}
	
	\subsection{Defining $\Phi^{-1}$}
    In order to demonstrate that $\Phi$ is a bijection, we construct a map $\Psi$ from tilings to paths.
    To begin with, we construct a matrix $C$ which captures a specific transformation of a path $\delta\in\P_k$.
    \begin{lemma}\label{lem:mcb}
		Consider the tiling $\M:=\Phi(\gamma,\delta)$.
        Without loss of generality, we may assume $(\gamma_1,\gamma_2,\ldots,\gamma_k)=I_k$.
        Then for all $j\in\Z$
        \[
        \begin{pmatrix}
            m_{1,j}\\
            m_{2,j}\\
            \vdots\\
            m_{k,j}
        \end{pmatrix}=C\delta_j
        \]
        where $C\in SL_k(\Z)$ has the following form
		\[
		C=
		\left[
		\begin{array}{ccc|c}
		0 & \cdots & 0 & 1\\
		\hline
		(-1)^{k-1}& & \textbf{0} & 0\\
		& \ddots & &\vdots\\
		*& &(-1)^{k-1} & 0
		\end{array}
		\right].
		\]
        Moreover, the entries on the lower-left of the matrix $C$ are determined uniquely by the $J$ matrices of $\gamma$.
	\end{lemma}
	
	\begin{proof}
		We construct the image of $\Phi(\gamma,\delta)$.
        Let $\{J_n\}_{n\in\Z}$ be the transition matrices for $\gamma$.
		Observe that, by construction, since $\gamma_i=e_i$ for all $i\in[k]$, we have
		\[
		m_{1\,j}=\det(e_1,\ldots,e_{k-1},\delta_j)=
		\det\left[
		\begin{array}{c|c}
		I_{k-1} &\begin{array}{c}
		\delta_{1\,j} \\
		\vdots \\
		\delta_{k-1\,j}
		\end{array}  \\
		\hline
		0\cdots 0 & \delta_{k\,j}
		\end{array}
		\right]=\delta_{k\,j}=C\delta_j,
		\]
        for $C$ as in the statement of the lemma.
		For $2\leq i\leq k$, by definition of the map $\Phi$ we have that
		\[
		m_{i\,j}= \det (e_i, e_{i+1}, \dots, e_k, \gamma_{k+1}, \dots, \gamma_{k+i-2}, \delta_j)=\
		\det\left[
		\begin{array}{c|c}
		\textbf{0} & P \\
		\hline
		I_{k-i+1} & *
		\end{array}
		\right]=\left((-1)^{i-1}\right)^{k-i+1}\det(P)
		\]
		where $P$ is the upper right $(i-1)\times (i-2)$ matrix of $\disp\prod_{n=1}^{i-2}J_n$ appended with the upper $(i-1)$ entries of $\delta_j$.
		Note that the determinant of $P$ is a linear combination of the entries of the last column, which is the upper $i-1$ entries of $\delta_j$.
		We focus on the coefficient of $\delta_{i-1\,j}$, which is given by the determinant of the upper left $(i-2)\times (i-2)$ matrix of $P$, call it $Q$.
		
		Note that the product $\disp\prod_{n=1}^{i-2}J_n$ may be written as
		\[
		\prod_{n=1}^{i-2}J_n=(\gamma_{i-1},\ldots,\gamma_{k+i-2})=
		\left[
		\begin{array}{c|c}
		\textbf{0} & Q\\
		\hline
		I_{k-i+2} & *
		\end{array}
		\right].
		\]
		Since each $J_n\in SL_k(\Z)$, the determinant of this product is one.  We may also calculate it in terms of $Q$ as
		\[
		1=\det\left(\prod_{n=1}^{i-2}J_n\right)=\left((-1)^{i-2}\right)^{k-i+2}\det(Q).
		\]
		This gives us that $\det(Q)=\left((-1)^{i}\right)^{k-i}$.

        Thus we have 
        \[
        m_{i\,j}=\det\left(
        \begin{array}{c|c}
            \textbf{0}&\begin{array}{c|c}
                Q&\begin{array}{c}
                    \delta_{1\,j}\\
                    \vdots\\
                    \delta_{i-2\,j}
                \end{array}\\
                \hline\\
                **&\delta_{i-1\,j}
            \end{array}\\
            \hline\\
            I_{k-i+1}&*
        \end{array}
        \right)
        \]
		Hence, the coefficient of $\delta_{i-1\,j}$ in $m_{i\,j}$ is
		\[
		\left((-1)^{i-1}\right)^{k-i+1}\det(Q)=\left((-1)^{i-1}\right)^{k-i+1}\left((-1)^{i}\right)^{k-i}=(-1)^{k-1}.
		\]
        Thus, we have the desired equation
        \[
        \begin{pmatrix}
            m_{1,j}\\
            m_{2,j}\\
            \vdots\\
            m_{k,j}
        \end{pmatrix}=C\delta_j.
        \]
        
        Furthermore, since $Q$ is determined by the $J$ matrices of $\gamma$, then so are the lower triangular entries of $C$.
	\end{proof}

    We use this to construct a map from tilings to pairs of paths modulo $SL_k(\Z)$.
    
	\begin{definition}\label{def:psi}
		Let $\M=(m_{ij})_{i,j\in\Z}$ be a tame $SL_k$-tiling and consider a map $\Psi$ 
        \begin{align*}
            \Psi:\slk&\rightarrow (\P_k\times \P_k)/SL_k(\Z)\\
            \M&\rightarrow(\gamma,\delta),
        \end{align*}
        where the paths $\gamma$ and $\delta$ are defined as follows.
		Let $\delta$ be the horizontal strip from $\M$ that is  $k$ entries tall centered at $\M_{[k],[k]}:=M$, which is to say that
			\[
			\delta_i=\begin{pmatrix}
			m_{1\,i}\\
			m_{2\,i}\\
			\vdots\\
			m_{k\,i}
			\end{pmatrix}.
			\]
		To define $\gamma$ we construct $\gamma_i$ for $i\in[k]$ and its $J$ matrices.
		  In order to construct the $J$ matrices, first let $\gamma' \in \P_k$ be such that $(\gamma')_{[k]}=I_k$ and its $J$ matrices are $V_i$ the vertical transition matrices of $\M$.    Then define $\gamma$ to have $J$ matrices $J_i=\widetilde{V_{i-k+2}}$, that is the shifted $J$ matrices of $\widetilde{\gamma'}$.  Now, it remains to define $\gamma_{[k]}$.
		Consider a tiling $\M'=\Phi(\gamma', \delta)$, and let $\gamma_{[k]}=M'_{11}(\delta_{[k]})^{-1}$, that is $\gamma_{[k]}$ is the matrix $C$ for the tiling $\M'$ from Lemma~\ref{lem:mcb}.
        	\end{definition}
	
	\begin{remark}\label{rem:jpreserving}
		By construction, under the map $\Psi$ the horizontal transition matrix $H_i$ of $\M$ equals the $i$-th $J$ matrix of $\delta$, and the $i$-th $J$ matrix of $\gamma$ is equal to $\widetilde{V_{i-k+2}}$ where $V_i$ is the $i$-th vertical transition matrix of $\M$. 
		    \end{remark}
	
	\begin{lemma}\label{lem:psi}
		The map $\Psi$ is well-defined.
	\end{lemma}
	\begin{proof}
	Let $\Phi(\M)=(\gamma,\delta)$, and we need to show that $\gamma,\delta\in\P_k$.
		Clearly, $\delta\in\P_k$, since each adjacent set of $k$ consecutive columns forms an adjacent $k\times k$ submatrix of $\M$.
        Furthermore, $\gamma\in\P_k$ since $\gamma_{[k]}=M'_{11}(\delta_{[k]})^{-1}$ is in $SL_k(\Z)$ and every other set of $k$ adjacent columns $(\gamma_i,\gamma_{i+1},\ldots,\gamma_{i+k-1})$ of $\gamma$ form a matrix which can be written as a product of $\gamma_{[k]}$ and $J$ matrices, which are all elements of $SL_k(\Z)$.
	\end{proof}

	We proceed to show that $\Phi$ and $\Psi$ are inverses of each other.
	
    \begin{proposition}\label{prop:phicircpsi}
		$\Phi\circ\Psi=\mathrm{id}_{\slk}.$
	\end{proposition}
	
	\begin{proof}
        Let $\M\in\slk$.
		Let $V_i$ and $H_i$ be the transition matrices of $\M$.
       Consider $\Psi(\M)=(\gamma,\delta)$.
        By Remark~\ref{rem:jpreserving}, $H_i$ is equal to the $i$-th $J$ matrix of $\delta$, while the $i$-th $J$ matrix for $\gamma$ is given by $\widetilde{V_{i-k+2}}$.
        Thus, by Proposition~\ref{cor:jmatrices}, $\Phi(\gamma,\delta)$ has as its $i$-th horizontal transition matrix $H_i$ and as its $i$-th vertical transition matrix 
        \[
        \widetilde{\widetilde{V_{i-k+2}}}=V_i
        \]
		by Remark~\ref{rem:tildeshift}.  
		
		It therefore suffices to check that $M:=M_{1\,1}$ is equal to
		$\left((\Phi\circ\Psi)(\M)\right)_{[k],[k]}$.
		Given $\Psi(\M)=(\gamma,\delta)$, and let $\gamma'$ be the path obtained from the vertical transition matrices of $\M$ as in Definition~\ref{def:psi}.  Note that by construction $(\gamma')_{[k]}=I_k$ and $\gamma=C\gamma'$ where $C$ is a matrix such that $\Phi(\gamma', \delta)_{[k],[k]}=C\delta_{[k]}$.  Moreover, by construction of $\Psi$, $\delta$ has the property that $M=\delta_{[k]}$.  Then we have 
		\[\Phi(\gamma, \delta)=\Phi(C\gamma', \delta)=\Phi(\gamma', C^{-1}\delta).\]
By Lemma~\ref{lem:mcb}, the $C$ matrix for $\Phi(\gamma, \delta)$ is the same as for $\Phi(\gamma', \delta)$, since it only depends on the transition matrices of the first path.  Thus, by Lemma~\ref{lem:mcb} applied to $\Phi(\gamma', C^{-1}\delta)$, we obtain 
\[\left((\Phi\circ\Psi)(\M)\right)_{[k],[k]}=\Phi(\gamma', C^{-1}\delta)_{[k],[k]}=C(C^{-1}\delta_{[k]})=\delta_{[k]}=M.\]
This shows the desired conclusion. 
        	\end{proof}
	
	\begin{proposition}\label{prop:psicircphi}
		$\Psi\circ\Phi=\mathrm{id}_{(\P_k\times\P_k)/SL_k(\Z)}$.
	\end{proposition}
	
	\begin{proof}
		Consider a pair of paths $\gamma$ and $\delta$ in $\P_k$.
		Let $(\gamma',\delta')=(\Psi\circ\Phi)(\gamma,\delta)$.
		By Proposition \ref{cor:jmatrices} and Remark \ref{rem:jpreserving}, we know that the $J$ matrices are the same for $\gamma$ and $\gamma'$ as well as $\delta$ and $\delta'$.
		After multiplying both $\gamma'$ and $\delta'$ by some matrix $A\in SL_k(\Z)$, which does not change the equivalence class of $(\gamma,\delta)$ in $(\P_k\times\P_k)/SL_k(\Z)$, we may assume that $\gamma=\gamma'$.
		Furthermore, we may set $(\gamma_1,\gamma_2,\ldots,\gamma_k)=I_k$.
		Since the $J$ matrices of $\delta$ and $\delta'$ are the equal, it suffices to show that
		\[
		(\delta_1,\delta_2,\ldots,\delta_k)=(\delta_1',\delta_2',\ldots,\delta_k').
		\]
		Since $(\gamma,\delta')=(\Psi\circ\Phi)(\gamma,\delta)$, by Proposition \ref{prop:phicircpsi}, we may apply $\Phi$ to both sides to get
        \[
        \Phi(\gamma,\delta)=\Phi(\gamma,\delta'):=\M.
        \]
        Let $C,C'$ be the matrices as in Lemma~\ref{lem:mcb} for $\Phi(\gamma,\delta), \Phi(\gamma,\delta')$ respectively. By Lemma \ref{lem:mcb} applies to $\M$, this implies that $C\delta_j=C'\delta_j'$ for all $j\in\Z$.
        Since $C, C'$ are uniquely determined by the $J$ matrices of $\gamma$, we have $C=C'$ and, since $C\in SL_k(\Z)$, this gives $\delta_j=\delta'_j$ for all $j\in\Z$.
			\end{proof}
	
	With both directions proven, we get our main result.
	
	\begin{theorem}\label{bijection}
		The map $\Phi$ given by
		\begin{align*}
		\Phi:(\P_k\times\P_k)/SL_k(\Z)&\rightarrow\slk\\
		(\gamma,\delta)&\mapsto \M=(m_{i\,j})_{i,j\in\Z},
		\end{align*} 
		where $m_{i\,j}=\det(\gamma_i,\ldots,\gamma_{i+k-2},\delta_j)$ is a bijection between tame $SL_k$-tilings and pairs of paths modulo the diagonal action by $SL_k(\Z)$.
	\end{theorem}
	\begin{proof}
		Follows from Proposition \ref{prop:phicircpsi} and Proposition \ref{prop:psicircphi}.
	\end{proof}

Next, we relate the bijection $\Phi$ to the correspondences appearing in \cite{s22} and \cite{br10}.
First, by restricting to the case $k=2$, we derive a result which is equivalent to the one shown in \cite{s22}.
In particular, it is easy to see that the maps $\Phi$ and $\Phi\Q$, see Definition~\ref{def:shortmap} and Theorem~\ref{thm:short}, are related as in the following commutative diagram
        \[
        \begin{tikzcd}
            (\P_2\times\P_2)/SL_2(\Z) \ar[r, "\Phi"] \ar[d,"\chi"] & \mathbb{SL}_2 \ar[d,"\chi'"]\\
            (\P\Q\times\P\Q)/SL_2(\Z) \ar[r,"\Phi\Q"] & \mathbb{SL}_2/\pm
        \end{tikzcd},
        \]
        where $\chi$ is a quotient map which identifies $(\gamma,\delta)$ and $(-\gamma,\delta)$.  Indeed, this follows because in $\P\Q$, the two paths of reduced rationals corresponding $\gamma$ and $-\gamma$ are equal.

	For general $k$,  the map $\Phi$ gives a nice description of the map $\xi^{-1}$ from Proposition \ref{prop:linearizationdata}.
	Let us first define another map $\rho$ which associates a path with its linearization data \cite[Equation 7]{br10}.
	
	\begin{definition}\label{def:rho}
		Define the function $\rho$ as follows.
		\begin{align*}
		\rho:(\P_k\times\P_k)/SL_k(\Z)&\rightarrow SL_k(\Z)\times \left(\Z^{1\times (k-1)}\right)^\Z\times\left(\Z^{1\times (k-1)} \right)^\Z\\
		(\gamma,\delta)&\mapsto M \times \{\left(j_{i\,2}, \ldots,j_{i\,k} \right)\}_{i\in\Z}\times \{\left(j'_{i\,2}, \ldots,j'_{i\,k} \right)\}_{i\in\Z}
		\end{align*}
		where the elements $j_{i\,q}$ and $j'_{i\,q}$ for $q\in[2, \dots, k]$ are given by the final columns of the $J$ matrices of $\widetilde{\gamma}$ and $\delta$ respectively, and $M= \Phi(\gamma, \delta)_{[k],[k]}$.
	\end{definition}

	Since linearization data is in bijection with $SL_k$-tilings by Proposition \ref{prop:linearizationdata}, we see that $\rho$ is also a bijection, hence it is invertible.
	While $\xi^{-1}$ is defined recursively, the map $\Phi$ is explicit.
	This makes the following result particularly noteworthy.
    
	\begin{corollary}
		The following diagram commutes.
		\[
				\begin{tikzcd}
					(\P_k\times\P_k)/SL_k(\Z) \ar[r,"\Phi"] \ar[d, "\rho"] & \slk\\
					SL_k(\Z)\times \left(\Z^{1\times (k-1)}\right)^\Z\times\left(\Z^{1\times (k-1)} \right)^\Z \ar[ur,"\xi^{-1}"]
				\end{tikzcd}
		\]		
	\end{corollary}
	
	\section{Applications}\label{sec:applications}

	In this section we discuss a few consequences of the bijection $\Phi$.  First, by placing certain restrictions on the paths we classify periodic $SL_k$-tilings as well as those coming from (infinite) friezes.  Then, we relate the construction of a dual of a tiling of Definition~\ref{def:dual} with the map $\Phi$, and show that it corresponds to applying the tilde operator on the two paths.  This gives a new interpretation of the notion of duality.

	\subsection{Periodicity}
	Here we show that periodic tilings correspond to periodic paths. 
	Recall the notion of periodicity given in Definition \ref{def:periodic}, which we now relate to periodicity of $J$ matrices. 
	\begin{lemma}\label{periodicj}
		For a path $\gamma$, if the $J$ matrices are $m$-periodic and $J_1J_2\cdots J_m=I_k$, then $\gamma$ is $m$-periodic.
	\end{lemma}
	\begin{proof}	
		By definition
		\[
		\left(\gamma_i,\ldots,\gamma_{i+k-1} \right)J_iJ_{i+1}\cdots J_{i+m-1}=\left(\gamma_{i+m},\ldots,\gamma_{i+m+k-1} \right).
		\]
		Thus, it suffices to show $J_i\cdots J_{i+m-1}=I_k$.
		Since $J_i=J_{i+m}$ for all $i\in\Z$, we may write the subscripts $\mathrm{mod}\,m$.
		Thus,
		\[
		J_iJ_{i+1}\cdots J_{i+m-1}=J_{i}J_{i+1}\cdots J_mJ_1J_2\cdots J_{i-1}
		\]
		Note that $J_1J_2\cdots J_m=I_k$ by assumption.
		Therefore, we multiply
		\begin{align*}
		J_{i}J_{i+1}\cdots J_mJ_1J_2\cdots J_{i-1}&=J_{i}J_{i+1}\cdots J_mJ_1J_2\cdots J_{i-1}\left(J_iJ_{i+1}\cdots J_mJ_m^{-1}J_{m-1}^{-1}\cdots J_i^{-1} \right)\\
		&=J_{i}J_{i+1}\cdots J_m\left(J_1J_2\cdots J_{i-1}J_iJ_{i+1}\cdots J_m\right)J_m^{-1}J_{m-1}^{-1}\cdots J_i^{-1}\\
		&=J_{i}J_{i+1}\cdots J_mJ_m^{-1}J_{m-1}^{-1}\cdots J_i^{-1}\\
		&=I_k.
		\end{align*}
	\end{proof}
	
	We may construct a similar result going from tilings to $J$ matrices.
	
	\begin{lemma}\label{periodicm}
		If $\M=\Phi(\gamma, \delta)$ is $m$-column periodic, then the $J$ matrices of $\delta$ are $m$-periodic and satisfy $J_iJ_{i+1}\cdots J_{i+m-1}=I_k$ for all $i\in\Z$.
	\end{lemma}

	\begin{proof}
		We know that we may get from $M_{1\,i}$ to $M_{1\,i+m}$ by multiplying by $J_iJ_{i+1}\cdots J_{i+m-1}$.
		Since $M_{1\,i}=M_{1\,i+m}$, as $\M$ is $m$-column periodic, this product must be the identity $I_k$.
		Furthermore, since column $i$ and column $i+m$ are the same for all $i\in\Z$, the same $J$ matrix must be used to transition between them, so $J_i=J_{i+m}$ for all $i\in\Z$.
	\end{proof}
	
	This gives us the following proposition.
	
	\begin{proposition}\label{prop:periodicity}
		Fix a path $\delta$.
		The following are equivalent.
		\begin{enumerate}
			\item The $J$ matrices of $\delta$ are $m$-periodic and $J_1J_2\cdots J_m=I_k$.
			\item The path $\delta$ is $m$-periodic.
			\item The tiling $\M:=\Phi(\gamma,\delta)$ for any path $\gamma$ is $m$-column periodic.
		\end{enumerate}
	\end{proposition}

	\begin{proof}
		The implication $(1)\Rightarrow (2)$ follows from Lemma~\ref{periodicj}, and the implication $(3)\Rightarrow (1)$ from Lemma~\ref{periodicm}.
		For the implication $(2)\Rightarrow (3)$, suppose $\delta$ is $m$-periodic.
		If $\delta$ repeats every $m$ entries, the resulting determinants $m_{i\,j}=\det(\gamma_i,\ldots,\gamma_{i+k-1},\delta_j)$ will also repeat every $m$ entries going across the columns.
		Thus, the resulting $\M$ will be $m$-column periodic, as desired.
	\end{proof}

	\begin{remark}\label{periodicitydual}
		The dual of Proposition~\ref{prop:periodicity} is also true.
		That is, the following are equivalent.
		\begin{enumerate}
			\item The $J$ matrices of $\gamma$ are $m$-periodic and $J_1J_2\cdots J_m=I_k$.
			\item The path $\gamma$ is $m$-periodic.
			\item The tiling $\M:=\Phi(\gamma,\delta)$ for any path $\delta$ is $m$-row periodic.
		\end{enumerate}
	\end{remark}
	
	\begin{remark}\label{skewperiodicity}
	Analogous arguments can also be applied to show an equivalence between skew $m$-periodic paths $\gamma$ and skew $m$-row periodic tilings $\M:=\Phi(\gamma,\delta)$. 
	\end{remark}

	Together, these results give us the necessary and sufficient conditions for periodicity in a tiling as in Definition \ref{def:periodic}.

	\begin{corollary}\label{blockperiodic}
		The paths $\gamma$ and $\delta$ are $m$ and $n$ periodic, respectively, if and only if $\M:=\Phi(\gamma,\delta)$ is $(m\times n)$-periodic.
	\end{corollary}

	\begin{proof}
		This follows from $(2)\Leftrightarrow(3)$ in Proposition \ref{prop:periodicity} and Remark \ref{periodicitydual}.
	\end{proof}

	\subsection{Friezes}
	In this section, our goal is to show that the bijection $\Phi$ has a simple restriction to infinite friezes, namely that friezes are in bijection with pairs of paths where both paths are identical.
	We first define some notation to clarify this.
	
	\begin{definition}
		Define the inclusion function $\iota$ as follows.
		\begin{align*}
			\iota:\P_k/SL_k(\Z)&\rightarrow (\P_k\times\P_k)/SL_k(\Z)\\
			(\gamma)&\mapsto(\gamma,\gamma).
		\end{align*}
	\end{definition}

	Note that the image of the map $\iota$ is a subset of $(\P_k\times\P_k)/SL_k(\Z)$, so we may use it as a  restriction of the domain of the map $\Phi$.
	Additionally, note that $\iota$ is a bijection on its own image, so $\Phi\circ\iota$ is injective.
	We write $\Phi_\iota$ for $\Phi\circ\iota$.
	We wish to show that this restriction on the domain corresponds to a restriction on the range to tilings resulting from infinite $SL_k$-friezes.  
	
	We begin by defining a tiling obtained from an infinite frieze, which is similar to the construction given in Definition \ref{def:totiling} for finite friezes.
	
	\begin{definition} \label{def:infinitetotiling}
		Let $F$ be an infinite frieze.
        We define a tiling $\M_F$ as follows.
		Let the rows of $F$ become the falling diagonals of a tiling by rotating the frieze $45^\circ$ clockwise and reflecting across a vertical line.
		Since $F$ is infinite, the right half of the tiling $\M_F$ is complete.  Moreover, this portion of the tiling contains all the information about vertical and horizontal transition matrices, which can then be used to uniquely construct the left half of $\M_F$.
			\end{definition}

	\begin{remark}
        Let $F$ be a finite frieze and let $F'$ denote its infinite extension.  That is, $F'$ is an infinite frieze obtained by extending $F$ periodically to an infinite number of rows. 
		We want to verify that the tiling $\M_{F'}$ in Definition \ref{def:infinitetotiling} coincides with $\M_F$ which results from Definition \ref{def:totiling}.
		Observe that both $\M_F$ and $\M_{F'}$ will have the same right half by construction.
		Given such a portion of an infinite tiling, there is a unique complete tiling which results by extending to the left along the falling diagonals \cite[Proposition 7]{br10}.
		Both $\M_F$ and $\M_{F'}$ are valid tilings, and must therefore be the same.
	\end{remark}

	It should be noted that the left half of the tiling $\M_F$, though recoverable as an infinite frieze, need not be a periodic extension of $F$ as with finite friezes.
	It also need not have the same properties as $F$, such as positivity.
	However, since this part is uniquely determined by $F$, we can restrict ourselves to talking about the right half of $\M_F$.
	Let $\FR_k$ denote the set of all $SL_k$-tilings resulting from infinite $SL_k$-friezes.
	
	\begin{theorem}\label{thm:friezes}
		The restriction of the map $\Phi$ given by
		\begin{align*}
			\Phi_\iota:\P_k/SL_k(\Z)&\rightarrow \FR_k\\
			\gamma&\mapsto \M=(m_{i\,j})_{i,j\in\Z}, 
		\end{align*}
		where $m_{i\,j}=\det(\gamma_i,\ldots, \gamma_{i+k-2},\gamma_j)$ is a bijection between tame $SL_k$-tilings from infinite $SL_k$-friezes and equivalence classes of paths.
	\end{theorem}

	\begin{proof}
		We first show that the map is well-defined.
		Let $\gamma\in\P_k$ and let $\M=\Phi_\iota(\gamma)$.
		Observe that the entries $m_{i\,j}$ in $\M$ where $i\in\Z$ and $j\in[i]^{k-1}$ are all zeros since they are given by determinants of matrices of the form $(\gamma_i,\ldots,\gamma_{i+k-2},\gamma_j)$ where $\gamma_j$ is the same as one of the previous columns.
		These constitute $k-1$ falling diagonals of zero entries.
		The next falling diagonal consisting of entries of the form $m_{i\,i+k-1}$ is all ones, since these are given by determinants of matrices formed by $k$ adjacent columns in $\gamma$.
		Thus, the right half of $\M$ is recoverable as an infinite frieze $F$.

		We now show that the map is a bijection.
		The map $\Phi_\iota$ inherits injectivity from the injectivity of $\Phi$ and $\iota$.
        For surjectivity, let $F$ be an infinite frieze (possibly the result of an extension of a finite frieze) and let $\M=\M_F$ be its corresponding tiling.  Then the tiling $\M$ has the following structure, where the leftmost zero in first row corresponds to the entry $m_{11}$.  Hence it has $k-1$ diagonals of zero's, followed to the right by diagonals of 1's and to the left by the diagonal of $(-1)^{k-1}$. 
        
        \[
        \begin{array}{ccccccccccccc}
        & \vdots & \vdots & \vdots & \vdots &\vdots&\vdots&\vdots&\vdots&\vdots\\
        \cdots & (-1)^{k-1} & 0 & 0  & \cdots & 0 & 1 & m_{1,k+1} & m_{1,k+2} & \cdots \\
        \cdots & m_{2,0} & (-1)^{k-1} & 0 & \cdots & &0 & 1 & m_{2,k+2} & m_{2,k+3} & \cdots\\
        \cdots & m_{3,0} & m_{3,1}& (-1)^{k-1}  & 0 & \cdots & &0& 1 & m_{3,k+3} & m_{3,k+2} & \cdots\\
        && \ddots & \ddots & \ddots &\ddots & &\ddots&\ddots &\ddots  &\ddots  \\
    & &\cdots &m_{k,k-1}&m_{k,k-2}& (-1)^{k-1} & 0 & \cdots & 0 & 1 & m_{k, 2k}&\cdots \\
    &&&\vdots & \vdots & \vdots & \vdots & \ddots &&&\vdots
        \end{array}
        \]
        
Let $\gamma$ be the horizontal strip of $\M$ in rows $[k]$, that is $\gamma_i = (m_{1i}, m_{2i}, \dots, m_{ki})^T$.  We claim that $\M=\Phi_{\iota}(\gamma)$.  By definition of $\Phi$ and the structure of $\gamma$ shown above, we see that the first row of $\Phi_{\iota}(\gamma)$ agrees with the first row of $\M$.  Let $V_1$ be the the vertical transition matrix of $\M$, see Definition~\ref{def:horizontalandverticaltransitions}, then $V_1^T\gamma$ is the horizontal strip of $\M$ in rows $2, \dots, k+1$.   Moreover, by Lemma~\ref{lem:invariance} we have $\Phi_{\iota}(\gamma)=\Phi_{\iota}(V_1^T\gamma)$.  By the same reasoning as above, it is easy to see that the second row of $\Phi_{\iota}(V_1^T\gamma)$ agrees with the second row of $\M$.   Now continuing in this way and multiplying by transposes of the vertical transition matrices $V_2, V_3, \dots$ and their inverses $V_{0}^{-1}, V_{-1}^{-1}, \dots$ the tiling $\M$ equals $\Phi_{\iota}(\gamma)$ for all rows. 
       	\end{proof}

	Since infinite friezes correspond to paths, we can further show that finite friezes correspond to periodic tilings and hence periodic paths. 
	Thus, we can make a further restriction.
	We denote by $\FR_{k,n}$ the set of all $SL_k$-tilings resulting from $SL_k$-friezes of type $(k,n)$.
	Let $\P_{k,n}$ denote the set of skew $n$-periodic paths in $\P_k$, see Definition~\ref{def:skewperiodic}.
	
	\begin{corollary}\label{cor:finiteperiodicfrieze}
		The restriction of $\Phi_\iota$ to skew $n$-periodic paths
		\[
		\Phi_\iota:\P_{k,n}/SL_k(\Z)\rightarrow\FR_{k,n}
		\] 
		is a bijection between skew $n$-periodic paths and tilings from $SL_k$-friezes of type $(k,n)$.
	\end{corollary}
	
	\begin{proof}
		The map is well-defined as a result of Remark~\ref{skewperiodicity} and Theorem \ref{thm:friezes}.
		Surjectivity follows from Remark \ref{rem:pluckertofrieze}, as these particular elements of the Grassmannian $A_F$ yield the skew periodic paths $\gamma=\phi_{A_F}$. 
		Injectivity is inherited from the injectivity of $\Phi_\iota$.
	\end{proof}

	\subsection{Duality}
	In this section we explore the connection between duality for tilings, see Definition \ref{def:dual}, and the map $\Phi$.  
	We show that the dual $\M^*$ of an $SL_k$-tiling $\M:=\Phi(\gamma,\delta)$ has a simple interpretation in terms of the map $\Phi$. 
	In particular, it is given by $\Phi(A\widetilde{\gamma},\widetilde{\delta})$, i.e. a shift of conjugacy class for $\gamma$ in $\P_k/SL_k(\Z)$ together with the tilde operator.
	We begin by presenting several results about the dual as it relates to Pl\"ucker friezes and their $J$ matrices.
	A special case of Theorem \ref{thm:dualtiling} for friezes was proven by Morier-Genoud, Ovsienko, Schwatz, and Tobachnikov \cite{mgost13} as well as Cordes and Roselle \cite{cr72}.
   
	We first give a lemma which shows that the dual preserves $J$ matrices up to the tilde operator.
	
	\begin{lemma}\label{lem:dualj}
		Let $\M=\Phi(\gamma,\delta)$ be an $SL_k$-tiling and let $\M^*=\Phi(\gamma^*,\delta^*)$ be its dual.
		\begin{itemize}
		\item[(a)] The sequences of transition matrices for $\widetilde{\gamma}$ and $\widetilde{\delta}$ coincide with those for $\gamma^*$ and $\delta^*$ respectively. 
				\item[(b)] $(\M^*)^*_{[k],[k]}=\M_{[k-1]^k, [k-1]^k}$.
		\end{itemize}
	\end{lemma}

	\begin{proof}
		Let $M^*:=\M^*_{[2k+1],[2k+1]}$.
		Note that $M^*$ is entirely determined by $\M_{[3k-1],[3k-1]}$.  Construct a matrix $A$ and a path $\phi(\gamma, \delta)=\phi_{4k,4k}(\gamma,\delta)$ as in Definition \ref{def:friezemap}.  
		Thus, in particular the first $4k$ columns of $A$ equal $\gamma_{[4k]}$ and for some $i$ the columns $[i]^{4k}$ of $A$ are equal $\delta_{[4k]}$.  Moreover, the path $\phi(\gamma,\delta)=\phi_A$ is the skew periodic extension of $A$.
		Let $\M'=\Phi(\phi(\gamma,\delta),\phi(\gamma,\delta))$ be the tiling resulting from this path. 
		By Lemma~\ref{lem:friezemapping} 
        \[
        \M_{[3k-1],[3k-1]}=\M'_{[3k-1],[i]^{3k-1}}=\begin{pmatrix}
            p_{[1]^{k-1}\,i}(A) & p_{[1]^{k-1}\,i+1} (A)& \cdots & p_{[1]^{k-1}\,i+3k-2}(A)\\
            p_{[2]^{k-1}\,i}(A) & p_{[2]^{k-1}\,i+1} (A)& \cdots & p_{[2]^{k-1}\,i+3k-2}(A)\\
            \vdots & \cdots & \ddots & \vdots\\
            p_{[3k-1]^{k-1}\,i}(A) & p_{[3k-1]^{k-1}\,i+1} (A)& \cdots & p_{[3k-1]^{k-1}\,i+3k-2}(A)
        \end{pmatrix}.
        \]
        Since the dual of a tiling is obtained by taking determinants of $(k-1)\times(k-1)$ submatrices, Proposition \ref{prop:pluckerdeterminant} implies that 
        \[
        \left( \M'_{[3k-1],[i]^{3k-1}} \right)^*=
        \begin{pmatrix}
            p_{[i]^{k-1}\,k-1}(A) & p_{[i+1]^{k-1}\,k-1} (A)& \cdots & p_{[i+2k-1]^{k-1}\,k-1}(A)\\
            p_{[i]^{k-1}\,k}(A) & p_{[i+1]^{k-1}\,k} (A)& \cdots & p_{[i+2k-1]^{k-1}\,k}(A)\\
            \vdots & \cdots & \ddots & \vdots\\
            p_{[i]^{k-1}\,3k-2}(A) & p_{[i+1]^{k-1}\,3k-2} (A)& \cdots & p_{[i+2k-1]^{k-1}\,3k-2}(A)
        \end{pmatrix}.
        \]
        Let $\M^*=\Phi(\gamma^*,\delta^*)$.  By Proposition~\ref{cor:jmatrices}, the first $J$ matrix of $\delta^*$ equals the first horizontal transition matrix for $\M^*$, which by the calculations above is the same as $i$-th vertical transition matrix for $\M'=\Phi(\phi(\gamma,\delta),\phi(\gamma,\delta))$.  By the same corollary this is the same as $i$-th $J$ matrix of the path $\widetilde{\phi(\gamma,\delta)}$, which in turn is the same as the first $J$ matrix of $\tilde\delta$ by construction of $\phi(\gamma,\delta)$.

        Similarly, since $\M^*=\Phi(\gamma^*,\delta^*)$ then Proposition~\ref{cor:jmatrices} implies that the second $J$ matrix of $\widetilde{\gamma^*}$ is the second vertical transition matrix for $\M^*$.  By looking at $\M^*$ and $\M'$ above, this is the same as $k$-th horizontal transition matrix for $\M'$, which in turn is the same as $k$-th $J$ matrix for $\phi(\gamma,\delta)$, which again is the same as the $k$-th $J$ matrix for $\gamma$ by construction of $\phi(\gamma,\delta)$.  Applying the tilde operator, we conclude that the second $J$ matrix of $\widetilde{\widetilde{ \gamma^*}}$ equals $k$-th $J$ matrix of $\widetilde\gamma$.  By Remark~\ref{rem:tildeshift} we see that this is the same as $k$-th $J$ matrix of $\gamma^*$.         Thus, the dual operator preserves $J$ matrices up to the tilde operator, which shows part (a).

        To show part (b) we calculate the $k\times k$ submatrix of $(\M^*)^*$. By Proposition \ref{prop:pluckerdeterminant} we have 
        \[
        (\M^*)^*_{[k],[k]}=
        \begin{pmatrix}
                p_{o([k-1]^{k-1},i+k-2)}(A) & p_{o([k-1]^{k-1},i+k-1)} (A)& \cdots & p_{o([k-1]^{k-1},i+2k-3)}(A)\\
                p_{o([k]^{k-1},i+k-2)} (A)& p_{o([k]^{k-1},i+k-1)} (A)& \cdots & p_{o([k]^{k-1},i+2k-3)}(A)\\
                \vdots & \vdots & \ddots & \vdots\\
                p_{o([2k-2]^{k-1},i+k-2)} (A)& p_{o([2k-2]^{k-1},i+k-1)} (A)& \cdots & p_{o([2k-2]^{k-1},i+2k-3)}(A)
            \end{pmatrix},
        \]
        which is just $\M_{[k-1]^k,[k-1]^k}$.
        This shows part (b).
	\end{proof}

	It should be noted that, although the dual operator preserves information about the transition matrices, it is less clear what happens to the $k\times k$ submatrix $\M^*_{[k],[k]}$ in relation to the initial matrix $\M_{[k],[k]}$.
    Thus, applying the dual to the tiling $\Phi(\gamma,\delta)$ may not preserve the conjugacy class of paths $(\widetilde\gamma,\widetilde\delta)$ in ${(\P_k\times \P_k)/ SL_k(\Z)}$.
	
	\begin{theorem}\label{thm:dualtiling}
		Let $\M=\Phi(\gamma,\delta)$ be an $SL_k$-tiling.
		Then its dual $\M^*=\Phi(A\widetilde{\gamma},\widetilde{\delta})$ for some $A\in SL_k(\Z)$.
	\end{theorem}

	\begin{proof}
		By Lemma \ref{lem:dualj}(a), the sequence of $J$ matrices associated with $\widetilde{\gamma}$ correspond to the sequence of vertical transition matrices of $\M^*$.
		Similarly, the sequence of $J$ matrices of $\widetilde{\delta}$ correspond to the sequence of horizontal transition matrices of $\M^*$.
		Thus, this maintains the linearization data up to the tilde operator.
        Multiplying $\widetilde{\delta}$ by an appropriate matrix $A\in SL_k(\Z)$ preserves the $J$ matrices by Lemma \ref{lem:invariance}.
        Doing so changes the central $k\times k$ adjacent submatrix ensuring that $\Phi(A\widetilde{\gamma},\widetilde{\delta})_{[k],[k]}=\M^*_{[k],[k]}$. 
        	\end{proof}

	As a corollary, we can also easily recover the fact the dual operator has the desired duality property.
	
	\begin{corollary}
		For any $SL_k$-tiling $\M$, $(\M^*)^*$ equals $\M$ up to a shift in indices by $k-2$. 
	\end{corollary}

    \begin{proof}
    Let $\M = \Phi(\gamma, \delta)$.  By Lemma~\ref{lem:dualj}(a) applying the dual operator twice to a tiling, amounts to applying the tilde operator twice to the paths $\gamma, \delta$.  By Remark~\ref{rem:tildeshift} applying the tilde operator twice shifts the $J$ matrices of paths by $k-2$.
 Similarly, by Lemma~\ref{lem:dualj}(b) applying the dual operator twice shifts the central matrix down and to the right by $k-2$. Therefore, since the $J$ matrices of paths are also shifted by $k-2$, we conclude that $(\M^*)^*$ equals $\M$ up to a shift in indices by $k-2$. 
    \end{proof}

	\section{Positivity}\label{sec:positivity}
	
	In this section we study sufficient conditions for a frieze to have positive entries.  In the case $k=2$, 
	Short proved several results about positive $SL_2$-tilings and friezes and their relations to paths using the geometry of the Farey graph \cite[Theorem 1.4]{s22}.
	Lacking a connection to geometry, we cannot use the same methodology.  Instead, in the case $k=3$ we relate positive friezes with $n$-periodic paths $\gamma$ that have alternating entries.  On the other hand, for small values of $k$ and $n$, namely when the cluster algebra $\mathcal{A}(k,n)$ on the associated Grassmannian $\text{Gr}(k,n)$ is of finite type,  we focus on studying positivity for friezes using Pl\"ucker coordinates.  In this case, we relate positive friezes with sequences of $J$ matrices that have alternating sign in their final column.

	\subsection{Paths with $k=3$}
	
Here we obtain a result about paths $\gamma$ which correspond to positive tilings from friezes when $k=3$.  We begin with the following definitions. 
	
	\begin{definition}
	We say that the vector $\gamma_i=(x_i, y_i, z_i)^T$ of a path $\gamma\in\P_3$ {\it alternates in sign} if $x_i, z_i$ are positive and $y_i$ is negative.  Similarly, we say that a path $\gamma$ with $(\gamma_1,\gamma_2,\gamma_3)=I_3$ alternates in sign if $\gamma_i$ of $\gamma$ alternates in sign for every $i\in\mathbb{Z}\setminus\{1,2,3\}$. 
	\end{definition}
	
	\begin{definition}
	We say that an (infinite) frieze $F$ is \emph{totally positive}, if the tiling $\M_F$ resulting from the frieze has positive entries except for the diagonals of zeros and the skew factor $(-1)^{k-1}$. That is, if $k$ is odd then the entries of $\M_F$ are all positive apart from the diagonals of zeros, and if $k$ is even then the entries of $\M_F$ are nonzero and have the same sign in between the diagonals of zeros.   
	\end{definition}

	Now, we examine two specific vectors of $\gamma$.
	
	\begin{lemma}\label{lem:alternatingbase}
		Let $\gamma\in\P_3$ be a path with $(\gamma_1,\gamma_2,\gamma_3)=I_3$.
		If $\M:=\Phi_\iota(\gamma)$ is a tiling from a totally positive frieze, then $\gamma_4$ and $\gamma_0$ alternate in sign.
			\end{lemma}
	
	\begin{proof}
		Since $(\gamma_0,\gamma_1,\gamma_2),(\gamma_2,\gamma_3,\gamma_4)\in SL_3(\Z)$, $\gamma_4$ and $\gamma_0$ are of the form
		\[
		\gamma_4=\begin{pmatrix}
		1\\
		a\\
		b
		\end{pmatrix},\quad\quad \gamma_0=\begin{pmatrix}
		c\\
		d\\
		1
		\end{pmatrix}.
		\]
		By definition of $\Phi_{\iota}$, we have the following entries of $\M$ which lie in the first nontrivial row:
		\[
		m_{1\,4}=\det\begin{pmatrix}
		1&0&1\\
		0&1&a\\
		0&0&b
		\end{pmatrix}=b,\quad\quad m_{3\,1}=\det\begin{pmatrix}
		0&1&1\\
		0&a&0\\
		1&b&0
		\end{pmatrix}=-a
		\]
		and
		\[
		m_{2\,0}=\det\begin{pmatrix}
		0&0&c\\
		1&0&d\\
		0&1&1
		\end{pmatrix}=c,\quad\quad m_{0\,3}=\det\begin{pmatrix}
		c&1&0\\
		d&0&0\\
		1&0&1
		\end{pmatrix}=-d.
		\]
		Since $\M$ is positive, so are $b,c,-a,$ and $-d$.
	\end{proof}
	
	Below, we examine arbitrary entries of $\gamma$.
	
	\begin{lemma}\label{lem:alternatinginduct}
		Let $\gamma\in\P_3$ be a path with $(\gamma_1,\gamma_2,\gamma_3)=I_3$ and entries $\gamma_i=(x_i,y_i,z_i)^T$.
		If $\M:=\Phi_\iota(\gamma)$ is a tiling from a totally positive infinite frieze, then the following hold for $i\in\mathbb{Z}\setminus \{0,1,2,3,4\}$.
		\begin{enumerate}
			\item $x_i>0$,
			\item $z_i>0$,
			\item $y_i<0$ if and only if $y_{i+1}<0$.
		\end{enumerate}

        Moreover, if $\gamma$ is $m$-periodic, the same holds for $\gamma_i$ where $i$ is not congruent to $0,1,2,3,4$ mod $m$.
	\end{lemma}
	
	\begin{proof}
		As in the proof of Lemma \ref{lem:alternatingbase}, the positivity of $m_{1\,i}$ and $m_{2\,i}$ give the positivity of $x_i$ and $z_i$ respectively.
		Furthermore, we have that
		\[
		m_{i\,3}=\det\begin{pmatrix}
		x_i&x_{i+1}&0\\
		y_i&y_{i+1}&0\\
		z_i&z_{i+1}&1
		\end{pmatrix}=x_iy_{i+1}-x_{i+1}y_i,\]
        \[
        m_{i\,1}=\begin{pmatrix}
		x_i&x_{i+1}&1\\
		y_i&y_{i+1}&0\\
		z_i&z_{i+1}&0
		\end{pmatrix}=y_iz_{i+1}-y_{i+1}z_i.
		\]
		By the positivity of $\M$ as well as $x_i$ and $z_i$, we have the following.
		If $y_i<0$, then $m_{i\,1}$ gives us that
		\[
		\frac{z_{i+1}}{z_i}<\frac{y_{i+1}}{y_i},
		\]
		so $y_{i+1}<0$.
		If $y_{i+1}< 0$, then $m_{i\,3}$ gives us that
		\[
		\frac{x_i}{x_{i+1}}<\frac{y_i}{y_{i+1}},
		\]
		so $y_{i}<0$.
	\end{proof}
	
	This allows us to make a general statement about the path $\gamma$.
	
	\begin{theorem}\label{thm:alternatingpath}
		Let $\gamma\in\P_3$ be a path with $(\gamma_1,\gamma_2,\gamma_3)=I_3$.
		If $\M:=\Phi_\iota(\gamma)$ is a tiling from a totally positive infinite frieze, then $\gamma$ alternates in sign.
        If $\M=\Phi_\iota(\gamma)$ is a tiling from a positive finite frieze with period $m$, then $\gamma_i$ alternate in sign for $i\not\equiv 1,2,3\mod{m}$. 
	\end{theorem}
	
	\begin{proof}
		We induct on $i$ from above starting with $i=4$ and below starting with $i=0$.
		Lemma \ref{lem:alternatingbase} gives the base cases, and Lemma \ref{lem:alternatinginduct} gives the inductive step.
	\end{proof}
	
	The converse of Theorem~\ref{thm:alternatingpath} is not generally true.
	Consider the following example of a path which alternates in sign
   \[
    \gamma=
	\left(
	\cdots,\begin{pmatrix}
	1\\
	0\\
	0
	\end{pmatrix},
	\begin{pmatrix}
	0\\
	1\\
	0
	\end{pmatrix},
	\begin{pmatrix}
	0\\
	0\\
	1
	\end{pmatrix},\left(
	\begin{array}{r}
	1\\
	-2\\
	1
	\end{array}\right),\left(
	\begin{array}{r}
	1\\
	-1\\
	1
	\end{array}\right),
          \left(
	\begin{array}{r}
	1\\
	-3\\
	2
	\end{array}\right),
    \left(
	\begin{array}{r}
	1\\
	-2\\
	1
	\end{array}\right),
	\begin{pmatrix}
	1\\
	0\\
	0
	\end{pmatrix},
	\begin{pmatrix}
	0\\
	1\\
	0
	\end{pmatrix},
	\begin{pmatrix}
	0\\
	0\\
	1
	\end{pmatrix}
	\cdots
	\right).
   \]
	This gives the following tiling:
	\[
    \Phi_\iota(\gamma)=
	\begin{array}{rrrrrrrrrrrrrr}
	0&0&1&1&1&2&1&0&0\\
	&0&0&1&1&1&1&1&0&0\\
	&&0&0&1&\textcolor{red}{-1}&0&2&1&0&0\\
	&&&0&0&1&0&-1&0&1&0&0\\
	&&&&0&0&1&1&-1&2&1&0&0\\
	&&&&&0&0&1&1&1&0&1&0&0,
	\end{array}
	\]
	which is clearly not a tiling from a positive frieze.
    Note the entry $m_{3\,5}=-1$ is given by
    \[
    -1=\left|
    \begin{array}{rrr}
         0&1&1\\
         0&-2&-3\\
         1&1&1 
    \end{array}
    \right|.
    \]

	\subsection{$J$ matrices of $SL_k$-friezes}
	In this section we study positivity for friezes by treating them as Pl\"ucker friezes evaluated at certain elements of the Grassmannian.   We begin by defining a new class of Pl\"ucker coordinates which correspond to entries of $J$ matrices and will play an important role in our discussion.
	
	\begin{definition}\label{def:semiconsecutive}
		A Pl\"ucker coordinate of the form $p_{o([i]^{k+1}\setminus\{j\})}$ where $i\in[n]$ and $j\in[i+1, \dots, i+k-1]$ i.e. a Pl\"ucker coordinate which consists of two consecutive runs separated by a gap of size one, is called \textit{semi-consecutive}.
	\end{definition}

	Note that by definition, consecutive Pl\"ucker coordinates are not semi-consecutive.
	Moreover, we have the following corollary to Proposition \ref{prop:jentries}.

	\begin{corollary}\label{cor:jentriessemi}
		The entries of $J$ matrices of tilings $\M_{\mathcal{F}_{(k,n)}}$ from Pl\"ucker friezes are semi-consecutive Pl\"ucker coordinates.  
        In particular, every semi-consecutive Pl\"ucker coordinate appears as an entry in some $J$ matrix of $\M_{\mathcal{F}_{(k,n)}}$.	
	\end{corollary}\label{cor:semi}

	Before continuing, we make several observations about Pl\"ucker coordinates for small values of $k$ and $n-k$. 
	
	\begin{remark}\label{rem:smalln}
		We consider two cases.
		\begin{enumerate}
			\item[(a)] If $k\geq n-3$, then all almost consecutive Pl\"ucker coordinates $p_I$ are either semi-consecutive or consecutive.
			With only three entries to exclude from $[n]$ to form $I$, there must either be a gap of one between sequences or all entries will be consecutive.

			\item[(b)] If $k\leq 3$, then all semi-consecutive Pl\"ucker coordinates $p_I$ are almost consecutive.
			With only three entries in $I$, at least one consecutive run will be of length one.
		\end{enumerate}
	\end{remark}


    Next, we want to further study the relationship between Pl\"ucker friezes and entries of their $J$ matrices.
   For ease of notation in the following results, we define quiddity vectors coming from the last column of $J$ matrices.
	
	\begin{definition}\label{def:quiddity}
		For an $SL_k$-frieze $F$, we call the vectors \[\textbf{q}_i=\begin{pmatrix}
			(-1)^{k-2}j_{i\,2}\\
			\vdots\\
			(-1)^{p}j_{i\,k-p}\\
			\vdots\\
			(-1)^0j_{i\,k}
		\end{pmatrix},\] 
		where $j_{i\,\ell}$ are the elements from the final column of the $i$-th horizontal transition matrix of $\M_F$, the \textit{quiddity vectors}.
		The sequence $(\textbf{q}_i)_{i\in\Z}$ is called the \textit{quiddity sequence} of $F$ or, equivalently, $\M_F$.
		We say a quiddity sequence is \textit{positive} if all of its entries are positive.
	\end{definition}

	This extends the traditional notion of the quiddity sequence in the case where $k=2$ to higher dimensions.
	Note that by Corollary~\ref{cor:semi}, the entries in these vectors correspond exactly to the semi-consecutive Pl\"ucker coordinates.

	\begin{lemma}\label{lem:n=k+3}
		Let $F$ be a frieze of type $(k,n)$ with $n-3\leq k$.
		Then $F$ is positive if the quiddity sequence is positive.
	\end{lemma}

	\begin{proof}
		We can realize the frieze ${F}$ as $F={\mathcal{F}_{(k,n)}(A)}$, that is the frieze obtained from a Pl\"ucker frieze evaluated at a matrix $A$.  Since the quiddity sequence is positive, this means all semi-consecutive Pl\"ucker coordinates of $A$ are positive.
		By part (a) of Remark \ref{rem:smalln}, all almost consecutive Pl\"ucker coordinates of $A$ in this case are also semi-consecutive.
		Since all entries of ${F}$ are given by almost consecutive Pl\"ucker coordinates of $A$, then all of them are also positive.
	\end{proof}
	
	\begin{lemma}\label{lem:k=2,3}
		Let $F$ be a frieze of type $(k,n)$ with $k\leq3$.
		Then the quiddity sequence is positive if $F$ is positive. 
	\end{lemma}

	\begin{proof}
		We can realize the frieze ${F}$ as $F={\mathcal{F}_{(k,n)}(A)}$, that is the frieze obtained from a Pl\"ucker frieze evaluated at a matrix $A$.   By part (b) of Remark \ref{rem:smalln}, we know that the entries of the quiddity vectors are almost consecutive Pl\"ucker coordinates of $A$, hence they are entries of $F$.
		Since all entries of $F$ are positive, then so is the quiddity sequence.
	\end{proof}
	
	\subsection{The Cases of $(4,7)$ and $(5,8)$}
	
	For some special pairs of $(k,n)$, we can make additional arguments.
	It should be noted that, while the representation theory of algebras that provide categorification of the cluster structure on $\mathcal{A}(k,n)$ does not feature directly in either of these arguments, it did serve as inspiration for the proofs.
	
	In what follows, we will always implicitly realize a frieze $F$ of type $(k,n)$ as ${F}={\mathcal{F}_{(k,n)}(A)}$, that is the frieze obtained from a Pl\"ucker frieze evaluated at some integer matrix $A$.  Recall that the consecutive Pl\"ucker coordinates of $A$ equal 1. Furthermore, to simplify notation, whenever we talk about the sign of a Pl\"ucker coordinate $p_I$, we always mean the sign of $p_I(A)$ for a fixed matrix $A$.  We also recall that Pl\"ucker coordinates satisfy Pl\"ucker relations given in \eqref{eqn:pluckerrel}.

	We begin with the case where $k=4$ and $n=7$.
	
	\begin{lemma}\label{lem:k=4}
		Let $F$ be a positive frieze of type $(4, n)$ where $n\leq 7$.
		Then the quiddity sequence is positive.
	\end{lemma}
	
	\begin{proof}
		Given that almost consecutive Pl\"ucker coordinates are positive, it suffices to show that the semi-consecutive Pl\"ucker coordinates are positive.
		It is only when $n\geq6$ that Pl\"ucker coordinates of the form $p_{o(m\,m+1\,m+3\,m+4)}$ are semi-consecutive, but not almost consecutive.
		We show that these are positive.
		Consider the Pl\"ucker relation with $I=\{m,m+1,m+3\}$ and $J =\{m+1,m+2,m+3,m+4,m+5\}$.
		This gives the equation
		\begin{multline*}
			0=p_{o(m\,m+1\,m+3)\,m+1}p_{o(m+2\,m+3\,m+4\,m+5)}-p_{o(m\,m+1\,m+3)\,m+2}p_{o(m+1\,m+3\,m+4\,m+5)}\\
			+p_{o(m\,m+1\,m+3)\,m+3}p_{o(m+1\,m+2\,m+4\,m+5)}-p_{o(m\,m+1\,m+3)\,m+4}p_{o(m+1\,m+2\,m+3\,m+5)}\\
			+p_{o(m\,m+1\,m+3)\,m+5}p_{o(m+1\,m+2\,m+3\,m+4)}.
		\end{multline*}
		Recall that Pl\"ucker coordinates with repeated entries are $0$ and consecutive Pl\"ucker coordinates are $1$.  
		Thus, we may simplify the above equation as follows.
		\begin{multline*}
		0=p_{o(m+1\,m+3\,m+4\,m+5)}-p_{o(m\,m+1\,m+3\,m+4)}p_{o(m+1\,m+2\,m+3\,m+5)}+p_{o(m\,m+1\,m+3\,m+5)}.
		\end{multline*}
		By assumption, $p_{o(m+1\,m+3\,m+4\,m+5)}$ and $p_{o(m+1\,m+2\,m+3\,m+5)}$ are positive since they are almost consecutive.
		Thus, 
		\[
		p_{o(m\,m+1\,m+3\,m+4)}>0\quad\text{if}\quad p_{o(m\,m+1\,m+3\,m+5)}\geq0.
		\]
		In the case of $n=6$, $m+5\equiv m-1\pmod 6$, so the latter is almost consecutive, hence positive.
		In the case of $n=7$, note that $m+6\equiv m-1\pmod 7$ and consider the Pl\"ucker relation with $I=\{m,m+1,m+3\}$ and $J =\{m,m+1,m+2,m+5,m+6\}$.
		Simplified as above, the resulting equation is
		\[
		0=-1-p_{o(m\,m+1\,m+3\,m+5)}+p_{o(m\,m+1\,m+3\,m+6)}p_{o(m\,m+1\,m+2\,m+5)}.
		\]
		Note that the final term is the product of almost consecutive Pl\"ucker coordinates, hence it is at least $1$.
		Therefore, $p_{o(m\,m+1\,m+3\,m+5)}\geq0$, as desired.
	\end{proof}

	For the case where $k=5$ and $n=8$, we first make claims regarding the case $k=3$ and $n=8$.
	
	\begin{remark}\label{rem:38to58}
		Note that $\Gr(5,8)\cong \Gr(3,8)$, where the isomorphism on their coordinate rings is given by $p_I\mapsto p_{I^c}$ where $I^c=[8]\setminus I$.
	\end{remark}

	Hence, for ease of notation we consider $\Gr(3,8)$.
    For the following, we assume that consecutive Pl\"ucker coordinates $p_{o(m\,m+1\,m+2)}=1$ and all Pl\"ucker coordinates are integers.

	\begin{lemma}\label{lem:135geq0}
		Consider $\Gr(3,8)$.  Suppose that all semi-consecutive Pl\"ucker coordinates are positive.
		Then Pl\"ucker coordinates of the form $p_{o(m\,m+2\,m+4)}\geq0$ for all $m$.
		Furthermore, for a fixed $m$, $p_{o(m\,m+2\,m+4)}=0$ if and only if $p_{o(m\,m+2\,m+3)}=p_{o(m+1\,m+2\,m+4)}=1$.
	\end{lemma}

	\begin{proof}
		Consider the Pl\"ucker relation where $I=\{m,m+2\}$ and $J=\{m+1,m+2,m+3,m+4\}$.
		This gives us the following equation:
		\begin{multline*}
		0=p_{o(m\,m+2)\,m+1}p_{o(m+2\,m+3\,m+4)}-p_{o(m\,m+2)\,m+2}p_{o(m+1\,m+3\,m+4)}\\
		+p_{o(m\,m+2)\,m+3}p_{o(m+1\,m+2\,m+4)}-
        p_{o(m\,m+2)\,m+4}p_{o(m+1\,m+2\,m+3)}.
		\end{multline*}
		Recall that consecutive Pl\"ucker coordinates are $1$ and that Pl\"ucker coordinates with repeated indices are $0$.
		We may therefore simplify the equation as
		\[
		p_{o(m\,m+2\,m+4)}=-1+p_{o(m\,m+2\,m+3)}p_{o(m+1\,m+2\,m+4)}.
		\]
		By assumption, the second term on the right side is a positive integer, therefore $p_{o(m\,m+2\,m+4)}\geq0$.
		Observe, additionally, that $p_{o(m\,m+2\,m+4)}=0$ if and only if the second term on the right side is $1$.  
		Since both Pl\"ucker coordinates in the product are positive integers, this occurs if and only if both are $1$.
	\end{proof}

	The same assumption also provides an additional result about a crucial orbit of Pl\"ucker coordinates.
	
	\begin{proposition}\label{prop:1314loop}
		Consider $\Gr(3,8)$.
        Suppose that all semi-consecutive Pl\"ucker coordinates are positive.
        \begin{itemize}
         \item[(a)] If for some fixed $q$, either $p_{o(q\,q+1\,q+5)}\geq0$ or $p_{o(q\,q+1\,q+4)}\geq0$, then $p_{o(m\,m+1\,m+5)}\geq0$ and $p_{o(m\,m+1\,m+4)}\geq0$ for all $m$.
         \item[(b)] If for some fixed $q$, either $p_{o(q\,q+1\,q+5)}=0$ or $p_{o(q\,q+1\,q+4)}=0$, then $p_{o(m\,m+1\,m+5)}=0$ and $p_{o(m\,m+1\,m+4)}=0$ for all $m$.
		\end{itemize}
	\end{proposition}

	\begin{proof}
        We proceed to prove two claims from which we may derive part (a).
        \begin{description}
            \item [\textbf{Claim 1a}] If $p_{o(m\,m+1\,m+5)}\geq0$, then $p_{o(m\,m+1\,m+4)}\geq0$.
            \item [\textbf{Claim 2a}] If $p_{o(m\,m+1\,m+4)}\geq0$, then $p_{o(m\,m+4\,m+7)}\geq0$.
        \end{description}
        \[\]
		Consider the Pl\"ucker relation where $I=\{m,m+1\}$ and $J=\{m+2,m+3,m+4,m+5\}$.
		This gives us the following equation.
		\begin{multline*}
		0=p_{o(m\,m+1)\,m+2}p_{o(m+3\,m+4\,m+5)}-p_{o(m\,m+1)\,m+3}p_{o(m+2\,m+4\,m+5)}\\
		+p_{o(m\,m+1)\,m+4}p_{o(m+2\,m+3\,m+5)}-
        p_{o(m\,m+1)\,m+5}p_{o(m+2\,m+3\,m+4)}.
		\end{multline*}
		Recall that consecutive Pl\"ucker coordinates are $1$.
		We may therefore simplify the equation as
		\begin{equation}\label{eq:a}
		p_{o(m\,m+1\,m+5)}=1-p_{o(m\,m+1\,m+3)}p_{o(m+2\,m+4\,m+5)}+p_{o(m\,m+1\,m+4)}p_{o(m+2\,m+3\,m+5)}.
		\end{equation}
		By assumption, all terms other than $p_{o(m\,m+1\,m+5)}$ and $p_{o(m\,m+1\,m+4)}$ are positive.
		This shows Claim 1a.
        
		Additionally, consider the Pl\"ucker relation where $I=\{m,m+4\}$ and $J=\{m,m+1,m+6,m+7\}$.
		This gives us the following equation.
		\begin{multline*}
		0=p_{o(m\,m+4)\,m}p_{o(m+1\,m+6\,m+7)}-p_{o(m\,m+4)\,m+1}p_{o(m\,m+6\,m+7)}\\
		+p_{o(m\,m+4)\,m+6}p_{o(m\,m+1\,m+7)}-p_{o(m\,m+4)\,m+7}p_{o(m\,m+1\,m+6)}.
		\end{multline*}
		By recalling that we are in $\Gr(3,8)$, this simplifies as
		\begin{equation}\label{eq:b}
		p_{o(m\,m+1\,m+4)}=-p_{o(m\,m+4\,m+6)}+p_{o(m\,m+4\,m+7)}p_{o(m\,m+1\,m+6)}.
		\end{equation}
		Note that, $p_{o(m\,m+1\,m+6)}$ is semi-consecutive and hence positive by assumption.  On the other hand, $p_{o(m\,m+4\,m+6)}$ is of the form $p_{o(m\,m+2\,m+4)}$, which is non-negative by  Lemma \ref{lem:135geq0}.
		Thus, we showed Claim 2a.
  
		Note that $p_{o(m\,m+4\,m+7)}$ is of the form $p_{o(m\,m+1\,m+5)}$ where the indices are shifted by $+1$.
		Suppose that $p_{o(q\,q+1\,q+5)}\geq0$ for some $q$.  
		By Claim 1a, this implies $p_{o(q\,q+1\,q+4)}\geq0$.
		By Claim 2a, this implies $p_{o(q\,q+4,q+7)}\geq0$ which in turn implies $p_{o(q-1\,q\,q+3)}\geq0$ by Claim 1a.
		One may iterate this process, concluding that $p_{o(m\,m+1\,m+5)}\geq0$ and $p_{o(m\,m+1\,m+4)}\geq0$ for all $m$.
		Alternatively, by starting from the assumption that $p_{o(q\,q+1\,q+4)}\geq0$ for some $q$, Claim 2a implies $p_{o(q\,q+4,q+7)}\geq0$ which in turn implies $p_{o(q-1\,q\,q+3)}\geq0$ by Claim 1a, and the iteration proceeds similarly.  This completes the proof of part (a). 
		
		To prove part (b), it similarly requires to show the following two statements. 
				\begin{description}
            \item [\textbf{Claim 1b}] If $p_{o(m\,m+1\,m+5)}=0$, then $p_{o(m+1\,m+2\,m+5)}=0$.
            \item [\textbf{Claim 2b}] If $p_{o(m\,m+1\,m+4)}=0$, then $p_{o(m\,m+1\,m+5)}=0$.
        \end{description}
		
		First, suppose that $p_{o(m\,m+1\,m+5)}=0$ for some $m$.  Then $p_{o(m+1\,m+2\,m+5)}\geq 0$ by part (a) and $p_{o(m+1\,m+5\,m+7)}\geq 0$ by Lemma~\ref{lem:135geq0}.
 Now, the following equation, obtained from \eqref{eq:b} by shifting indices $+1$,  
		\[
		p_{o(m+1\,m+2\,m+5)}=-p_{o(m+1\,m+5\,m+7)}+p_{o(m\,m+1\,m+5)}p_{o(m+1\,m+2\,m+7)}
		\]
		implies that $p_{o(m+1\,m+2\,m+5)}=0$.   This shows Claim 1b.
					
		Now suppose that $p_{o(m\,m+1\,m+4)}=0$ for some $m$.  Then by part (a) we conclude that $p_{o(m\,m+1\,m+5)}\geq 0$.  Then equation \eqref{eq:a} implies that $p_{o(m\,m+1\,m+5)}= 0$, since $p_{o(m\,m+1\,m+3)}, p_{o(m+2\,m+4\,m+5)}$ are semi-consecutive and positive by assumption.  This shows Claim 2b. 
	\end{proof}

	For the following proposition, we examine Pl\"ucker coordinates of certain types in the case $\Gr(3,8)$.
	We will use this result later to make claims about the corresponding coordinates in the case $\Gr(5,8)$.
	
	\begin{proposition}\label{prop:237leq0}
		Consider $\Gr(3,8)$.
		Suppose that consecutive Pl\"ucker coordinates are $1$ and semi-consecutive Pl\"ucker coordinates are positive.  Consider a subset of Pl\"ucker coordinates
		\[S=\{p_{o(m\,m+1\,m+5)}, p_{o(m\,m+1\,m+4)} \mid m\in \mathbb{N}\}. \]
				If there exists a non-positive element in $S$, then all elements of $S$ are $0$ and all semi-consecutive Pl\"ucker coordinates are 1.
			\end{proposition}

	\begin{proof}
	We begin by presenting a list of specific Pl\"ucker relations for $\Gr(3,8)$ which we will use throughout the proof. Each equation is presented with its sets $I$ and $J$.
	
	\begin{itemize}
			\item[(i)] If $I=\{3,7\}$ and $J=\{1,2,3,4\}$, then
			\[
			p_{3\,4\,7}=p_{2\,3\,7}p_{1\,3\,4}-p_{1\,3\,7}.
			\]
			
			\item[(ii)] If $I=\{3,7\}$ and $J=\{4,5,6,7\}$, then
			\[
			p_{3\,6\,7}p_{4\,5\,7}=p_{3\,5\,7}p_{4\,6\,7}-p_{3\,4\,7}.
			\]
			
			\item[(iii)] If $I=\{3,7\}$ and $J=\{2,3,4,5\}$, then
			\[
			p_{2\,3\,7}=p_{3\,4\,7}p_{2\,3\,5}-p_{3\,5\,7}.
			\]
\end{itemize}

		To prove the statement we may, without loss of generality, pick a particular Pl\"ucker coordinate of each form from the collection $R$ since all equations hold after a shift in indices.
		For the following, we refer to the list of relations (i)-(iii) from above. 
		
		We set $m=2$ and for the first case we suppose that $p_{2\,3\,7}\leq 0$.  
		Consider relation (i).
		By assumption, $p_{2\,3\,7}\leq 0$ and $p_{1\,3\,4}>0$.
		By Lemma \ref{lem:135geq0}, $p_{1\,3\,7}\geq0$.
		Thus, $p_{3\,4\,7}\leq0$. 
		Consider relation (ii).
		By assumption, $p_{4\,5\,7},p_{4\,6\,7}\geq0$.
		By Lemma \ref{lem:135geq0}, $p_{3\,5\,7}\geq0$.
		By the above, $p_{3\,4\,7}\leq0$.
		Thus, $p_{3\,6\,7}\geq0$.
		Note that $p_{3\,6\,7}$ is of the form $p_{o(n\,n+1\,n+5)}$, so we may apply Proposition \ref{prop:1314loop}(a).
        Thus, $p_{3\,4\,7}\geq 0$.
		Combining with above results, we conclude $p_{3\,4\,7}=0$ as desired.  By Proposition~\ref{prop:1314loop}(b) we conclude that all elements of $S$ are zero.
		
		We now consider the second case where $p_{3\,4\,7}\leq 0$.
		Consider relation (iii).
		By Lemma \ref{lem:135geq0}, $p_{3\,5\,7}\geq0$.
		Thus, $p_{2\,3\,7}\leq0$.
		As above, relation (ii) and Proposition \ref{prop:1314loop}(b) give us the same result that all elements of $S$ are zero.
		
		It remains to show that all semi-consecutive Pl\"ucker coordinates are 1. 
		Relation (i) gives $p_{1\,3\,7}=0$, hence $p_{1\,3\,8}=p_{1\,2\,7}=1$ by Lemma \ref{lem:135geq0}.  Similarly, by shifting the indices in relation (i) we conclude that all semi-consecutive Pl\"ucker coordinates are 1. 
			\end{proof}

	This leads us to a statement concerning the case of $\text{Gr}(5,8)$.
	
	\begin{proposition}\label{prop:k=5}
		Let $F$ be a positive frieze of type $(5,8)$.
		Then the quiddity sequence is positive with the exception of the frieze $F$ of all $1$'s. 
	\end{proposition}

	\begin{proof}
		Recall from Remark \ref{rem:38to58} that there is an isomorphism between $\Gr(3,8)$ and $\Gr(5,8)$.
		Consider the collection $S$ from Proposition \ref{prop:237leq0}.
		Performing the complement operation $p_{I^c}$ on elements of $S$ yields all semi-consecutive Pl\"ucker coordinates for $\Gr(3,8)$ that are not almost consecutive.  Similarly, performing it on the semi-consecutive Pl\"ucker coordinates of $\Gr(3,8)$ yields all almost consecutive Pl\"ucker coordinates in $\Gr(5,8)$ that are not consecutive.  Hence, Proposition \ref{prop:237leq0} can be restated for $\Gr(5,8)$ as saying that if all almost consecutive Pl\"ucker coordinates are positive then all semi-consecutive Pl\"ucker coordinates are positive with the single exception of when all almost consecutive Pl\"ucker coordinates are 1. 
				
		Since $F$ is a positive frieze, then all almost consecutive Pl\"ucker coordinates in $\Gr(5,8)$ are positive. Hence by above all semi-consecutive entries in $\Gr(5,8)$ are positive with the single exception when every entry of $F$ is 1.  Since, semi-consecutive coordinates are precisely the entries of the quiddity sequence, this yields the desired result. 
			\end{proof}

	\subsection{Gale dual and positivity result}
	Using the notion of the Gale dual of an $SL_k$-frieze, we are able to extend correspondences between positive friezes and positive quiddity sequences to other small values of $k$ and $n$.	
	
	We begin by recalling the notion of the Gale dual following the work of Morier-Genoud, Ovsienko, Schwartz, and Tobachnikov \cite{mgost13}.

	\begin{definition}\label{def:galetransform}\cite[Definitions 4.1.3 and Proposition 5.2.1]{mgost13}
		Let $F$ be a tame frieze of type $(k,n)$ with entries $a_{ij}$ as in Definition~\ref{def:frieze}. 
        Let $\alpha_{ij}$ be the determinant of the $j\times j$ diamond in $F$ whose left corner is $a_{i\,i}$ for $j\leq k$.
		The \textit{Gale dual of $F$}, denoted $F^\G$, is the tame frieze of type $(n-k, n)$ with entries $\alpha_{ij}$.
	\end{definition}

	We make the key observation that the entries of the Gale dual of $F$ are semi-consecutive minors of Pl\"ucker coordinates, which are precisely the entries of the quiddity sequence for $F$.

	Additionally, we refer to the following result from Ovsienko who proves a connection between $3d$-dissections of polygons and positive quiddity sequences for friezes in the case $k=2$.  In particular, the following statement is a direct consequence of the main theorem of \cite{o18}.
		
	\begin{theorem}\label{thm:2,10}\cite{o18}
		Let $k=2$ and $n\leq 9$.
		Then the Gale dual restricts to a bijection on positive friezes.
	\end{theorem}

	We can now state a partial positivity result for $SL_k$- friezes.
	
	\begin{theorem}\label{thm:positivitybij}
		Let $F$ be a tame $SL_k$-frieze of type $(k,n)$ satisfying one of the following conditions:
		\begin{enumerate}
			\item $k=2$ and $n\leq9$,
			\item $k=3,6$ and $n\leq 8$,
			\item $k=4$ and $n\leq 7$,
			\item $k=5$ and $n\leq 7$,
			\item $k=5$ and $n=8$ with the exception of the frieze of all ones where the quiddity vectors are all $(
			1,
			0,
			0,
			0,
			1
			)^T$.
		\end{enumerate}
		Then $F$ is positive if and only if its quiddity sequence is positive.
	\end{theorem}

	\begin{proof}
		We break this into cases.
		\begin{description}

		\item[Case (1) with $k=2, n\leq 9$]
		The forward direction follows from Lemma \ref{lem:k=2,3}.
		We show the backward direction.
		Suppose that the semi-consecutive Pl\"ucker coordinates are all positive.
		Then $F^\G$ is a positive frieze of type $(n-k, n)$, so  by Theorem \ref{thm:2,10} the frieze $F$ is positive as well.
		
		\item[Case (3) with $k=4, n\leq 7$]
		The forward direction follows from Lemma \ref{lem:k=4}.
		The backward direction follows from Lemma \ref{lem:n=k+3}.
		
		\item[Cases (4) and (5) with $k=5, n\leq 8$]
		The backward direction follows from Lemma \ref{lem:n=k+3}.
		The forward direction follows from Proposition \ref{prop:k=5} when $n=8$.
		For $n=7$, suppose that the frieze $F$ is positive.
		Then, by Theorem \ref{thm:2,10}, its Gale dual $F^\G$ is a positive frieze where $k=2$.
		This means that semi-consecutive Pl\"ucker coordinates of $F$ are also positive.  	             Hence the quiddity sequence of $F$ is positive. 
		
		\item[Case (2) with $k=3, n\leq 8$]
		The forward direction follows from Lemma \ref{lem:k=2,3}.
		We prove the backward direction.
		For $n\leq 6$, this follows from Lemma \ref{lem:n=k+3}.
		For $n=7,8$, suppose that the semi-consecutive Pl\"ucker coordinates are all positive.
		Then $F^\G$ is a positive frieze where $n=7$ and $k=4$ or $n=8$ and $k=5$.
		By Cases (3) and (4)-(5), respectively, the resulting frieze has positive semi-consecutive Pl\"ucker coordinates with a single exception, hence its Gale dual has positive entries.
		Thus, the initial frieze $F$ is also positive.
		For the exception in Case (5), note that the Gale dual of the all ones frieze in $\Gr(3,8)$ is not a positive tiling in $\Gr(5,8)$ as the consecutive Pl\"ucker coordinates of the form $p_{o(n\,n+1\,n+5)}=p_{o(n\,n+1\,n+4)}=0$.
		Thus, this case does not come into play here.
		
		\item [Case (2) with $k=6, n\leq 8$]
		The backward direction follows from Lemma \ref{lem:n=k+3}.
		We prove the forward direction.
		Suppose that the frieze $F$ is positive.
		Then $F^\G$ is a positive frieze where $k=2$ by Theorem~\ref{thm:2,10}.
		This means that the semi-consecutive Pl\"ucker coordinates of $F$ are also positive.
		\end{description}
	\end{proof}

We remark that the correspondence of Theorem~\ref{thm:positivitybij} between positive friezes and positive quiddity sequences generally fails for higher values of $n$.  Hence, for general $k$ and $n$, one needs to develop another notion in this higher-dimensional setting that captures positivity.

	As an application, we make a connection to the enumeration of friezes. 
	It was conjectured by Fontaine and Plamondon in \cite{FP} that there are $26952$ positive friezes of type $(3,8)$, which was recently shown by Zhang in \cite{Zhang}.  By Theorem \ref{thm:positivitybij}, positive friezes of type $(3,8)$ are in bijection with the positive friezes of type $(5,8)$ with a single exception in the $(5,8)$ case not appearing in the $(3,8)$ case.  Thus, we immediately obtain the following result.

	\begin{corollary}
		There are $26953$ positive friezes of type $(5,8)$.
	\end{corollary}

\bibliographystyle{alpha}
\bibliography{SLk tilings}
\end{document}